\documentclass[11pt,a4paper, leqno]{article}
\usepackage{amsmath}
\usepackage{amsfonts}
\usepackage{amssymb}
\usepackage{comment}
\usepackage{hyperref}
\usepackage{enumitem,pstricks,xy,pst-node}
\xyoption{all}
\usepackage{amsthm}
\usepackage{booktabs} 
\usepackage[color]{changebar} 
\cbcolor{red} 
\usepackage{mathrsfs} 
\usepackage{verbatim} 
\usepackage[disable] 
{todonotes}
\usepackage[normalem]{ulem} 
\usepackage{caption} 
\usepackage{nicefrac} 
\usepackage{enumitem} 
\usepackage{mathtools}

\newtheorem*{gpthm}{Generic Perfection Theorem \cite{EN67}}
\newtheorem{theorem}{Theorem}[section]
\newtheorem{lemma}[theorem]{Lemma}
\newtheorem{proposition}[theorem]{Proposition}

\newtheorem{definition}[theorem]{Definition}
\newtheorem{question}[theorem]{Question}
\theoremstyle{definition}

\newtheorem{remark}[theorem]{Remark}
\newtheorem*{ack}{Acknowledgements}

\theoremstyle{remark}
\newtheorem{example}[theorem]{Example}

\newcommand{\PP}{\mathbb{P}}

\newcommand{\Gr}{\mathbf{Gr}}
\newcommand{\OGr}{\mathbf{OGr}}
\newcommand{\SGr}{\mathbf{SGr}}
\newcommand{\OF}{\mathbf{OF}}
\newcommand{\IGr}{\mathbf{IGr}}
\newcommand{\Fl}{\mathbf{F}}
\newcommand{\HOM}{\mathcal{H} om}
\newcommand{\EXT}{\mathcal{E} xt}
\newcommand{\dys}{D_Y(s)}
\newcommand{\Odys}{\cO_{D_Y(s)}}

\def\cT{{\mathcal T}}
\def\cU{{\mathcal U}}
\def\zero{\mathscr{Z}}

\def\PP{\mathbf P}

\def\CC{\mathbb{C}}

\def\QQ{\mathbb{Q}}
\def\ZZ{\mathbb{Z}}

\def\cE{{\mathcal E}}

\def\cV{{\mathcal V}}
\def\cW{{\mathcal W}}
\def\cO{{\mathcal O}}
\def\cF{{\mathcal F}}
\def\cQ{{\mathcal Q}}
\def\cL{{\mathcal L}}
\def\ra{\rightarrow}

\def\fg{\mathfrak{g}}
\def\fh{\mathfrak{h}}

\def\fsl{\mathfrak{sl}}
\def\fso{\mathfrak{so}}
\def\fe{\mathfrak{e}}

\def\af1{\mathbf{aff}_1}
 
\DeclareMathOperator{\HS}{HS}
\DeclareMathOperator{\Ann}{Ann}
\DeclareMathOperator{\im}{Im}
\DeclareMathOperator{\rank}{rank}

\DeclareMathOperator{\Aut}{Aut}

\DeclareMathOperator{\Hom}{Hom}

\DeclareMathOperator{\Sym}{Sym}
\DeclareMathOperator{\Sing}{Sing}
\DeclareMathOperator{\codim}{codim}

\DeclareMathOperator{\HHH}{H}

\DeclareMathOperator{\depth}{depth}
\DeclareMathOperator{\pd}{pd}

\date{}
\AtEndDocument{\noindent Vladimiro Benedetti\\
Institut de Math\'ematiques de Marseille, UMR 7373\\Aix-Marseille Universit\'e, CNRS, Centrale Marseille\\13453 Marseille, France.\\
vladimiro.BENEDETTI@univ-amu.fr\\
\rule{1pt}{0pt}\\
Sara Angela Filippini\\
Department of Mathematics, Imperial College London\\
South Kensington Campus\\
London SW7 2AZ, UK.\\
s.filippini@imperial.ac.uk\\
\rule{1pt}{0pt}\\
Laurent Manivel\\
Institut de Math\'ematiques de Toulouse, UMR 5219\\
Universit\'e de Toulouse, CNRS, UPS IMT\\
F-31062 Toulouse Cedex 9, France.\\
manivel@math.cnrs.fr\\
\rule{1pt}{0pt}\\
Fabio Tanturri\\
Laboratoire Paul Painlev\'e, UMR CNRS 8524\\
Universit\'e de Lille\\
59655 Villeneuve d'Ascq CEDEX, France.\\
Fabio.Tanturri@math.univ-lille1.fr}
\author{Vladimiro Benedetti\thanks{Institut de Math\'ematiques de Marseille, UMR 7373, Aix-Marseille Universit\'e, CNRS, Centrale Marseille, 13453 Marseille, France.} \and Sara Angela Filippini\thanks{Department of Mathematics, Imperial College London, South Kensington Campus, London SW7 2AZ, UK.} \and Laurent Manivel\thanks{Institut de Math\'ematiques de Toulouse, UMR 5219, Universit\'e de Toulouse, CNRS, UPS IMT F-31062 Toulouse Cedex 9, France.} \and Fabio Tanturri\thanks{Laboratoire Paul Painlev\'e, UMR CNRS 8524, Universit\'e de Lille, 59655 Villeneuve d'Ascq CEDEX, France.}}
\title{Orbital degeneracy loci II: Gorenstein orbits}
\begin{document}
\maketitle

\begin{abstract}
In \cite{BFMT} we introduced orbital degeneracy loci as generalizations of degeneracy loci
of morphisms between vector bundles. Orbital degeneracy loci can be constructed from any 
stable subvariety of a representation of an algebraic group. In this paper we show that their canonical bundles can be 
conveniently controlled in the case where the affine coordinate ring of the subvariety is Gorenstein. 
We then study in a systematic way the subvarieties obtained as orbit closures in representations with finitely many 
orbits, and we determine the canonical bundles of the 
corresponding orbital degeneracy loci in the Gorenstein cases. Applications are given to the construction of low
dimensional varieties with negative or trivial canonical bundle.
\end{abstract}

\section{Introduction}

In \cite{BFMT} we introduced orbital degeneracy loci and studied their first properties. An orbital degeneracy locus is the locus of points of a manifold where a given section of a vector bundle degenerates to a
fixed subspace of its total space, defined fiberwise by a $G$-stable closed subvariety of a representation of an algebraic group $G$. This notion generalizes the classical degeneracy loci of a morphism between two vector bundles, supported on the points of the manifold in which the morphism has bounded rank.

One of our main motivations to introduce orbital degeneracy loci is to construct 
new Fano varieties and new manifolds with trivial canonical
bundle of low dimension, which have been gaining more and more interest in view of many recent developments in algebraic geometry. In this perspective, it is absolutely crucial to control the canonical sheaf, 
which is a priori far from obvious; even for classical degeneracy loci this control is easy 
only when the bundles have the same rank. In \cite{BFMT} we considered the most favorable situation, which occurs when the $G$-stable subvariety one is interested in can be desingularized by a crepant \emph{Kempf collapsing}. This happens indeed for determinantal loci
in square matrices, which is what allowed to construct, for example, certain interesting Calabi--Yau threefolds as classical degeneracy loci (\cite{Bertin09, KK10}). In \cite{BFMT} we 
exhibited other crepant Kempf collapsings, and used them to construct dozens 
of new interesting Calabi--Yau or Fano manifolds of dimension three and four.

Crepant Kempf collapsings, however, do not always exist and are in general not easy to find. Moreover, some interesting Calabi--Yau threefolds have also been constructed as symmetric or skew-symmetric degeneracy loci (\cite{Tonoli04, Kan12, GP01}), although no crepant Kempf collapsing is known in those contexts. The main purpose of this paper is to give a precise description of the canonical bundle of an orbital degeneracy locus associated to a $G$-stable closed subvariety having Gorenstein affine coordinate ring, a condition which seems 
to be the weakest under which a good control of this canonical bundle is possible. Using the classical Generic Perfection Theorem, from a free $G$-equivariant resolution of the affine coordinate ring we deduce a complex of vector bundles that resolves the structure sheaf of an orbital degeneracy locus constructed inside a smooth algebraic variety; in the Gorenstein case such a locally free resolution yields the canonical sheaf, which is exactly what makes symmetric or skew-symmetric degeneracy loci easier to handle.

The above result leads to considering Gorenstein $G$-stable subvarieties of $G$-varieties. 
A wide source of interesting examples  is provided by the so-called 
parabolic representations (or type I theta groups), which are special representations
coming from $\mathbb{Z}$-gradings of complex simple Lie algebras. In such representations there
exist only finitely many orbits, and the orbit closures are resolved by (non-necessarily crepant) Kempf collapsings. They have attracted considerable attention over the years, and huge quantities of information about them have been accumulated; the Kempf--Lascoux--Weyman geometric technique \cite{Weyman2003}, for instance, often allows to determine a free resolution of the orbit closures, and thus to decide whether they are Gorenstein. The second goal of the paper is to enlarge the list of cases we can use to construct interesting algebraic varieties as orbital degeneracy loci: we mostly build upon a series of paper by Kra\'skiewicz and Weyman \cite{KW12, KW13, KWE8} to study all the Gorenstein parabolic orbit closures and to characterize the corresponding orbital degeneracy loci together with their canonical bundle. We use them to exhibit several examples of fourfolds with trivial canonical bundle; a remarkable outcome is the realization of some $4$-dimensional irreducible holomorphic symplectic manifolds as orbital degeneracy loci, which can be shown to be Hilbert schemes of points on K3 surfaces (Remark \ref{remIHS}).

The paper is structured as follows: in Section \ref{locallyFreeRes}, after briefly recalling the definition of an orbital degeneracy locus associated to a $G$-stable subvariety of a $G$-representation, we prove Theorem \ref{locallyfree}, which provides under mild assumptions a locally free resolution of the structure sheaf of an orbital degeneracy locus. We use this result to characterize the canonical bundle of the orbital degeneracy locus in the Gorenstein case (Theorem \ref{GorensteinControl}); we then describe the relation between crepant Kempf collapsings and Gorenstein rings. Section \ref{GorParOrb} includes a 
reminder on parabolic representations, and discusses the classical types, while Section \ref{sec.exceptionals} studies the parabolic representations of exceptional types; for both cases we characterize the orbital degeneracy loci associated to Gorenstein orbit closures and provide several examples of projective fourfolds with trivial canonical bundle.

We remark that all this machinery might very well be used to construct new interesting subvarieties beside Fano varieties or varieties with trivial canonical bundle. For instance, some examples of curves of low genus can be easily realized as orbital degeneracy loci. 
Moreover, even though in this article we will place ourselves in a complex algebraic framework, considering a more general situation (e.g. when the ambient manifold is non-necessarily complex or algebraic) is certainly feasible.
Our hope is that  our constructions will allow to describe new families of varieties or manifolds having interesting features, and we plan to do so in our subsequent work.

\begin{ack}
	The authors wish to thank Jerzy Weyman for stimulating discussions, and for communicating 
	\cite{KWE8} to them. They also thank the referee for useful comments. This work has been mostly carried out in the framework of the Labex Archim\`ede (ANR-11-LABX-033) and of the A*MIDEX project (ANR-11-IDEX-0001-02), funded
	by the "Investissements d'Avenir" French Government program managed by
	the French National Research Agency. The second author is supported by the Engineering and Physical Sciences Research Council Programme Grant ``Classification, Computation, and Construction: New Methods in Geometry'' (EP/N03189X/1). The fourth author is supported by the Labex CEMPI (ANR-11-LABX-0007-01).
\end{ack}

\section{Locally free resolutions of orbital degeneracy loci}
\label{locallyFreeRes}

In this section we briefly recall the notion of an orbital degeneracy locus (ODL for short) $\dys$ and we construct an exact complex of locally free sheaves resolving its sheaf of regular functions $\Odys$. This yields a simple way to
describe the canonical bundle of an ODL, when  $Y$ is assumed to be Gorenstein.

We will say that a variety (or a scheme) $X$ is Cohen--Macaulay (respectively, Gorenstein) if $\cO_{X,x}$ is a Cohen--Macaulay (respectively, Gorenstein) ring for any $x \in X$.
The ground field will always be the complex numbers. 

\subsection{ODL and their resolutions of singularities}
\label{ODLandRes}

We present here for completeness a quick overview of the definition and main properties of the ODL introduced in \cite{BFMT}.

Let $G$ be an algebraic group acting on a variety $Z$. For any $G$-principal bundle 
$\cE$ over a smooth complex algebraic variety $X$, there is an associated bundle $\cE_Z$ over $X$ with fiber 
$Z$, defined as the quotient of $\cE\times Z$ by the equivalence relation 
$(eg,z)\sim (e,gz)$ for any $g\in G$. In particular, if $V$ is a $G$-module, then
$\cE_V$ is a vector bundle over $X$, with fiber $V$; if $Y$ is a $G$-stable subvariety of $V$, $\cE_Y$ is a subfibration of $\cE_V$ over $X$, with fiber $Y$.

\begin{definition}[ODL]
	Suppose that $V$ is a $G$-module and $Y$ a $G$-stable subvariety of $V$. 
	Let $s$ be a global section of the vector bundle $\cE_V$. Then the $Y$-degeneracy locus of $s$, denoted by $D_Y(s)$, is 
	the scheme defined by the Cartesian diagram
	\begin{equation}
		\label{cartdiag}
		\xymatrix{\ar@{}[dr] |{\square}
			\cE_Y \ar[r] & \cE_V  \\
			D_Y(s) \ar[u] \ar@{^{(}->}[r] & X \ar[u]_-s }
	\end{equation}
	Its support is $\{x\in X, \; s(x)\in \cE_Y \subset \cE_V\} = s^{-1}(\cE_Y)$.
\end{definition}
When no confusion may arise, we will simply call $\dys$ the ODL associated to the section $s$.

\begin{proposition}[{\cite[Proposition 2.3]{BFMT}}]
	\label{codimNormality}
	Suppose that $\cE_V$ is globally generated and let $s$ be a general section. Then $\Sing(\dys)=D_{\Sing (Y)}(s)$ and
	\[
	\codim_X \dys =\codim_V Y, \quad \codim_{\dys} \Sing(\dys)= \codim_{Y} \Sing(Y).
	\]
	If moreover $Y$ is normal (respectively, has rational singularities), then also $\dys$ is normal (respectively, has rational singularities).
\end{proposition}

A nice situation occurs when $Y$ admits a resolution of singularities given by a Kempf collapsing. 
This means that there exist a parabolic subgroup $P$ of $G$ and a $P$-submodule $W \subset V$ such that the associated homogeneous vector bundle $\cW$ over the projective variety $G/P$, which is a subbundle of the trivial bundle $G/P\times V$, has the following property: the restriction to $\cW$ of the projection $G/P\times V \rightarrow V$ is birational and surjective onto $Y$.
\begin{equation}
	\label{kempfCollapsing}
	\xymatrix @C=2pc @R=0.4pc{
		G/P\times V \ar[dd] & \rule{1pt}{0pt}\cW \ar@{_{(}->}[l] \ar[dr] \ar[dd]^-{p_W} \\
		& & G/P\\
		V & \rule{1pt}{0pt}Y \ar@{_{(}->}[l] 
	}
\end{equation}
\begin{theorem}[\cite{Kempf76}]\label{KempfInventiones}
	If $G$ is connected and $\cW$ is completely reducible, then $Y$ has rational singularities and Cohen--Macaulay affine coordinate ring.
\end{theorem}

The situation illustrated in diagram \eqref{kempfCollapsing} can be globalized over $X$ as follows. From the $G$-principal bundle 
$\cE$ over $X$ we construct a variety $\cF_W$ as the quotient of $\cE\times G\times W$ by the 
equivalence relation $(e,h,w)\sim (eg^{-1}, ghp^{-1},pw)$, for $g\in G$ and $p\in P$. 
The projection $p_{12}$ over the
first two factors induces a map $\cF_W\rightarrow \cE_{G/P}$ which makes $\cF_W$ a vector 
bundle over $\cE_{G/P}$, with fiber $W$. Moreover the map $(e,h,w)\mapsto (e,hw)$ induces 
a proper morphism $\cF_W\rightarrow \cE_V$, whose image is $\cE_Y$. This gives a relative 
version over $X$  of the morphism $p_W:\cW\rightarrow Y$.
In particular $\cF_W\rightarrow \cE_Y$
is birational when $p_W$ is birational. Note moreover that $\cF_V\simeq\theta^*\cE_V$,
if $\theta : \cE_{G/P}\rightarrow X$ is the projection map.
The inclusion $\cF_W \subset \cF_V$ induces the following short exact sequence of vector bundles on $\cE_{G/P}$:
\[
\xymatrix{
	0 \ar[r] &
	\cF_W \ar[r]&
	\cF_V \ar[r]^-\eta&
	Q_W \ar[r]&
	0
}.
\]

Consider now a global section $s$ of the vector bundle $\cE_V$ on $X$. Pulling it back to
$\cE_{G/P}$ and modding out by $\cF_W$, we get a global section $\tilde{s}:=\eta \circ \theta^*(s)$ of 
$Q_W$, whose zero locus maps to the $Y$-degeneracy locus of $s$:
$$\theta (\zero(\tilde{s})) = D_Y(s).$$
The relative version of \eqref{kempfCollapsing} is illustrated by the following commutative diagram:
\begin{equation}
	\label{relversion}
	\xymatrix @C=2pc @R=0.8pc{
		\cF_V \ar[dd] & \rule{1pt}{0pt}\cF_W \ar@{_{(}->}[l] \ar[dr]^-{p_{12}} \ar[dd] \\
		& & \cE_{G/P} \ar[dd]^-{\theta} & \ar@{_{(}->}[l] \ar[dd]^-{\theta'} \rule{1pt}{0pt}\zero(\tilde{s})\\
		\cE_V & \rule{1pt}{0pt}\cE_Y \ar@{_{(}->}[l] \ar[rd] \\
		&& X & \rule{1pt}{0pt}D_Y(s) \ar@{_{(}->}[l]
	}
\end{equation}

\begin{proposition}[{\cite[Proposition 2.3]{BFMT}}]
	\label{propressing}
	Suppose that $p_W: \cW \rightarrow Y$ is a birational Kempf collapsing, $\cE_V$ is globally generated and $s$ is a general section. Then the restricted projection $\theta' : \zero(\tilde{s})\longrightarrow D_Y(s)\subset X$ is a resolution of singularities. 
\end{proposition}

\subsection{Locally free resolutions of ODL}

In the settings of the previous section, the advantage of $Y$ being resolved by a Kempf collapsing \eqref{kempfCollapsing} is that we can get some information about the free $A$-resolution and the syzygies of the ideal $I_Y$ defining $Y$, $A=Sym(V^*)$ being the coordinate ring of $V$. This method is usually referred to as the Kempf--Lascoux--Weyman method (or simply Weyman's method) of computing syzygies via resolutions of singularities. It is developed in full generality in \cite{Weyman2003}, to which we refer for more background.

The aim of what follows is to show that a relative version of a ($G$-equivariant) free resolution of $A/I_Y$ yields an exact complex of locally free $\cO_X$-modules resolving $\Odys$. We already recalled 
that, given a $G$-principal bundle $\cE$ over a variety $X$, there is an induced  
functor $\cE_{-}$ from the category of $G$-representations to the category of vector bundles on $X$. 
This functor is exact and monoidal.

\begin{theorem}
	\label{locallyfree}
	With the above notation, suppose that $A/I_Y$ is a Cohen--Macaulay ring. Let $F_\bullet$ be a $G$-equivariant graded free $A$-resolution of minimal length of $A/I_Y$.
	Assume that $s \in \HHH^0(X, \cE_V)$ is a section such that $D_Y(s)$ has the expected codimension $\codim_V Y$.
	Then $\cE_{F_\bullet}$ is a locally free resolution of $\cO_{\cE_Y}$ as $\cO_{\cE_V}$-module, and $s^*(\cE_{F_{\bullet}})$ is a locally free resolution of $\cO_{D_Y(s)}$. In particular, $D_Y(s)$ is Cohen--Macaulay.
\end{theorem}
\begin{proof}
	Let $m$ be the length of $F_\bullet$. We can write the terms of $F_\bullet$ as
	\begin{equation}
		\label{termsAres}
		F_i = \bigoplus_{j} V_{{i,j}} \otimes A(-i-j),
	\end{equation}
	for some $G$-modules $V_{{i,j}} \subset (V^*)^{\otimes i+j}$. By functoriality we obtain a complex $\cE_{F_\bullet}$, whose terms are
	\[
	\cE_{F_i} = \bigoplus_{j} \cE_{V_{{i,j}}} \otimes_{\cO_X} \cO_{\cE_V},
	\]
	where $\cE_{V_{{i,j}}} \subset (\cE_V^*)^{i+j}$. Since the maps are induced by the maps of $F_\bullet$, this complex is exact as a complex of $\cO_{X}$-modules. When regarded as a complex of $\cO_{\cE_V}$-modules, it remains exact and its terms can be written as $\cE_{F_i}=\oplus_{j} \pi^* \cE_{V_{{i,j}}}$, being $\pi:\cE_V \rightarrow X$ the natural projection.
	
	We claim that $\cE_{F_\bullet}$ resolves $\cO_{\cE_Y}$. Indeed, locally on an open subset $U \subset X$ the $G$-equivariant maps $F_{i+1} \ra F_{i}$ induce maps $(\cE_{F_{i+1}})_U \simeq U \times F_{i+1} \ra U \times F_{i} \simeq (\cE_{F_{i}})_U$. 
	The cokernel of the last but one map of $\cE_{F_\bullet}$ injects into $\cO_{\cE_{V}}$, thus it is the ideal sheaf of a subvariety of $\cE_{V}$. Since locally on $U$ the last map defines $I_Y$, we get that $\cE_{F_\bullet}$ is a locally free resolution of $ \cO_{\cE_Y}$ of length $m$.
	
	Pulling the complex $\cE_{F_\bullet}$ back to $X$ via $s$ yields a complex whose $i$-th term is $\oplus_{j} \cE_{V_{{i,j}}}$. Since $s^*\cO_{\cE_Y}=\cO_{\dys}$, it remains to prove that $s^*\cE_{F_\bullet}$ is exact, for which we use the following classical theorem:
	
	\begin{gpthm}
		Let $S$ be a commutative ring and $R$ be a polynomial ring over $S$. Let $G_\bullet$ be a free $R$-resolution of a module $M$ of length $m=\depth \Ann_R(M)$ and assume that $M$ is free as an $S$-module. Let $\phi:R\rightarrow R'$ be a ring homomorphism such that $m=\depth \Ann_{R'}(M\otimes_R R')$. Then $G_\bullet \otimes_R R'$ is a free $R'$-resolution of $M \otimes_R R'$.
	\end{gpthm}
	
	In a neighbourhood of any $x \in X$, $\cO_{\cE_V}$ is a polynomial ring over $\cO_{X,x}$ with as many variables as the rank of $\cE_V$. The localization of $\cE_{F_\bullet}$ at each point $p \in \cE_V$ such that $\pi(p)=x$ is a free $\cO_{\cE_V,p}$-resolution of $\cO_{\cE_Y,p}$ of length $\pd_A A/I_Y$. Let $J$ be the ideal of $\cE_Y$ inside $\cO_{\cE_V,p}$. We have $m:=\pd_A A/I_Y= \codim_A I_Y$ since $A/I_Y$ is a Cohen--Macaulay ring; by hypotheses, also $\cO_{X,x}$ and $\cO_{\cE_V,p}$ are Cohen--Macaulay rings, hence 
	\[
	m= \codim_{\cO_{\cE_V,p}} J = \depth(J,\cO_{\cE_V,p}) = \depth_{\cO_{\cE_V,p}}(\Ann_{\cO_{\cE_V,p}} (\cO_{\cE_Y,p})).
	\]
	By hypothesis, $\dys$ is equidimensional. As $\cO_{X,x}$ is Cohen--Macaulay, 
	for $s^*:\cO_{\cE_V,p}\rightarrow \cO_{X,x}$ we have that $\depth \Ann_{\cO_{X,x}}(\cO_{\cE_Y,p}\otimes_{\cO_{\cE_V,p}} \cO_{X,x})=\codim_X \dys$, which is $m$ by hypothesis.
	The exactness of $s^*(\cE_{F_\bullet})$ then follows from the Generic Perfection Theorem. In particular, this implies that $\dys$ is Cohen--Macaulay: locally around $y \in \dys$, $(s^*(\cE_{F_\bullet}))_y$ is a free $\cO_{X,y}$-resolution of length $m$. Since $\cO_{X,y}$ is a regular local ring, we conclude by \cite[Corollary 19.15]{Eisenbud95}.
\end{proof}

\begin{remark}
	For an empty ODL $\dys$, we will conventionally say that $\dys$ has the expected codimension $\codim_V Y$. This way, Theorem \ref{locallyfree} holds even for an empty ODL, and $s^*(\cE_{F_\bullet})$ is an exact complex of vector bundles on $X$. This complex can be used, for concrete examples, to ensure a posteriori that an ODL is non-empty.
	
	In the same spirit, we will say that the singular locus of $\dys$ has a certain codimension even though it is allowed to be empty.
\end{remark}

\begin{remark}
	So far we have assumed that $X$ is a smooth complex algebraic variety: this convention will remain throughout the whole paper. However, we remark that many of the definitions and constructions concerning ODL can be mimicked in much more general situations, e.g., for a non-necessarily algebraic manifold $X$.
	As a matter of fact, one can also contemplate working on another field than the complex numbers.
	
	Another possible generalization concerns singular algebraic varieties: it turns out that many of the results which hold true for smooth varieties can be reproduced in the case of varieties with mild singularities, usually Cohen--Macaulay, even though one gets a good control on the geometry of an ODL only outside the singularities of $X$. When working in more general situations, however, one should be careful about dealing with $G$-principal bundles and locally free sheaves.
\end{remark}

\begin{example}
	\label{classicalDL}
	Consider two vector spaces $V_e$ and $V_f$ of dimension $e, f$ respectively, let $G=GL_e \times GL_f$ and $V$ be the space of $f \times e$ matrices $V_e^* \otimes V_f$, viewed as the natural $G$-representation. Suppose $e \geq f$ and consider $Y=Y_{f-1}$, the subvariety of matrices of corank at least $1$. Such subvariety is resolved by the total space of the vector bundle $\HOM(V_e,\cU)$ over $\Gr(f-1,V_f)$, $\cU$ being the tautological rank $f-1$ subbundle.
	
	Weyman's method \cite[(6.1.6)]{Weyman2003} yields a $G$-equivariant graded minimal free $A$-resolution $C_\bullet$ of $A/I_Y$, best known as the Eagon--Northcott complex, whose first term is the free module $C_0=A$ and $i$-th term is
	\begin{equation*}
		C_i= \wedge^{f+i-1} V_e \otimes \wedge^f V_f^* \otimes Sym^{i-1} V_f^* \otimes A(-f-i+1).
	\end{equation*}
	
	Let $\cE$ be a $G$-principal bundle over a variety $X$ and $s \in \HHH^0(X,\cE_V)$; in this case, the ODL $\dys$ turns out to be the first degeneracy locus of the morphism $\varphi_s:\cE_{V_e} \rightarrow \cE_{V_f}$ associated to $s$. Suppose that $\dys$ has the expected codimension. Then Theorem \ref{locallyfree} yields a locally free resolution of $\Odys$ whose $i$-th term, for $i>0$, is the $\cO_X$-module
	\[
	s^*(\cE_{C_\bullet})_i=\wedge^{f+i-1} \cE_{V_e} \otimes \det \cE_{V_f}^* \otimes Sym^{i-1} \cE_{V_f}^*.
	\]
	This is nothing more than the well-known Eagon--Northcott complex for vector bundles resolving the classical degeneracy locus $D_{\varphi_s}$, see e.g.\ \cite[Theorem B.2.2]{Lazarsfeld2004}).
\end{example}

\begin{example}
	\label{partDec}
	Let $V_6$ be a complex vector space of dimension six, $G=GL(V_6)$, $V=\wedge^3 V_6$ and $Y$  the subvariety of partially decomposable tensors, see \cite[\textsection 3]{BFMT}. A $G$-equivariant resolution of $A/I_Y$ has been computed via Weyman's method in \cite[\textsection 5]{KW12}, as $Y$ corresponds to the closure of the orbit $\cO_2$ for the $G$-representation associated to $(E_6, \alpha_2)$. Let us denote by $S_{(\lambda)}$ the Schur functor associated to the Young diagram $\lambda$; for instance, $S_{(1^3,0^3)}V_6$ denotes the third exterior power of $V_6$, while $S_{(1^6)}V_6$ is its determinant. Then, the equivariant resolution of $A/I_Y$ is given by
	\begin{multline*}
		A \leftarrow S_{(2^3,1^3)} {V_6}^* \otimes A(-3) \leftarrow S_{(3,2^4,1)} {V_6}^* \otimes A(-4) \leftarrow \\
		\leftarrow S_{(4,3^4,2)} {V_6}^* \otimes A(-6) \leftarrow S_{(4^3,3^3)} {V_6}^* \otimes A(-7) \leftarrow S_{(5^6)} {V_6}^* \otimes A(-10) \leftarrow 0.
	\end{multline*}
	Let $E$ be a rank 6 vector bundle over a variety $X$, let $\cE$ be its frame bundle and let $s$ be a global section of $\cE_V=\wedge^3 E$. Suppose that $\dys$ has the expected codimension five. Then Theorem \ref{locallyfree} yields a complex of vector bundles over $X$ resolving $\Odys$ given by
	\begin{multline*}
		\cO_X \leftarrow S_{(1^3,0^3)} E^* \otimes \det E^* \leftarrow S_{(2,1^4,0)} E^* \otimes \det E^* \leftarrow \\ S_{(2,1^4,0)} E^* \otimes (\det E^*)^2 \leftarrow S_{(1^3,0^3)} E^* \otimes (\det E^*)^3 \leftarrow (\det E^*)^5 \leftarrow 0.
	\end{multline*}
\end{example}

\subsection{Twisted degeneracy loci}
\label{TwistedDeg}
A slight modification of the definition of ODL leads to the so-called \emph{twisted degeneracy loci}, already introduced in \cite[\textsection 3.2.2]{BFMT}. As usual, let $V$ be a $G$-representation, $Y \subset V$ a $G$-invariant affine subvariety and $\cE$ a $G$-principal bundle over a variety $X$. Let $L$ be a line bundle over $X$ and consider $s \in \HHH^0(\cE_V \otimes L)$. Then we can define the twisted ODL $D_{Y}(s)$ as
\begin{equation*}
	\xymatrix{\ar@{}[dr] |{\square}
		\cE_Y \otimes L \ar[r] & \cE_V \otimes L \\
		D_Y(s) \ar[u] \ar@{^{(}->}[r] & X \ar[u]_-{s}
	}
\end{equation*}

Twisted ODL can be seen as ordinary ODL. Indeed, let us denote by $\cL$ the frame bundle of $L$; then $\cE':=\cE \times_X \cL$ is a $G':=G \times GL_1$-principal bundle over $X$. If $W$ is the natural $GL_1$-representation, one has $\cE'_{V \otimes W}=\cE_V \otimes L$. The subvariety $Y \otimes W \subset V \otimes W$ is $G'$-stable, and the ODL associated to a section of $\cE'_{V \otimes W}$ turns out to be the same as the twisted ODL introduced above.

\begin{remark}
	Suppose that $Y$ is a cone, which is true for instance when it is resolved by a Kempf collapsing. From the minimal free resolution of $Y$ we can easily write, by applying Theorem \ref{locallyfree}, a locally free resolution of the twisted ODL, which will depend on the choice of $L$. Indeed, if $A$ denotes the affine coordinate ring of $V$ and the terms of an $A$-resolution of $I_Y$ are as in \eqref{termsAres}, then there exists a locally free resolution of the twisted ODL whose terms are
	$
	\oplus_j \cE_{V_{{i,j}}} \otimes L^{(i+j)}.
	$
\end{remark}

\subsection{The canonical bundle of an ODL in the Gorenstein case}

In Theorem \ref{locallyfree} we showed that a locally free resolution of an ODL can be constructed from the free $A$-resolution of the (Cohen--Macaulay) coordinate ring of $Y \subset V$, where $A$ denotes the coordinate ring of $V$. When $A/I_Y$ is further assumed to be a Gorenstein ring, its minimal $A$-resolution is self-dual and its last term is free of rank one, see e.g.\ \cite[Corollary 21.16]{Eisenbud95}. This has a beautiful consequence for the locally free resolution of Theorem \ref{locallyfree}, which leads to the following result.

\begin{theorem}
	\label{GorensteinControl}
	Let $V$ be a $G$-representation, $A$ its affine coordinate ring and $Y \subset V$ a $G$-invariant subvariety with ideal $I_Y$ of codimension $c$ such that $A/I_Y$ is a Gorenstein ring. 
	
	Let $X$ be a variety, $\cE$ a $G$-principal bundle on $X$ and $s$ a section of $\cE_V$. Suppose that the ODL $D_Y(s)$ has the expected codimension $c$. Then $D_Y(s)$ is Gorenstein and its dualizing sheaf $\omega_{D_Y(s)}$ is the restriction of some line bundle on $X$.
	\begin{proof}
		Since by Theorem \ref{locallyfree} the ODL $\dys$ is Cohen--Macaulay, its dualizing complex is a sheaf, which can be computed as
		\[
		\omega_{\dys} = \EXT^c_{\cO_X}(\Odys,\omega_X),
		\]
		$\omega_X$ being the dualizing sheaf of $X$. Consider a minimal free $A$-resolution of $A/I_Y$, which is self-dual by hypothesis; Theorem \ref{locallyfree} yields a self-dual locally free resolution $\cF_\bullet$ of $\cO_{\dys}$. Since $\EXT^i_{\cO_X}(\Odys,\omega_X)=0$ for all $i < c$, $\omega_{\dys}$ is resolved by $\cF_\bullet^* \otimes \omega_X$. The last term $\cF_c$ is a line bundle over $X$, hence $\omega_{\dys} \simeq \cF_c^* \otimes {\omega_X}_{|\dys}$. This implies that $\omega_{\dys}$ is locally free of rank one and $\dys$ is Gorenstein.
	\end{proof}
\end{theorem}

\begin{remark}
	\label{dualcanonical}
	The ODL $\dys$ is not assumed to be normal. When this is the case the dualizing sheaf $\omega_{\dys}$ coincides with the canonical sheaf $K_{\dys}$, see e.g.\ \cite[\textsection 2.3]{KovacsSchwedeSmith}. It is not in general locally free of rank one; this happens exactly when it is Cartier and $\dys$ is Gorenstein.
	
	In practice, $\dys$ turns out very often to be normal for free. On the one hand, a Cohen--Macaulay variety with singularities in codimension at least two is normal; on the other hand, Proposition \ref{codimNormality} gives us some sufficient conditions for the normality of $\dys$.
\end{remark}

\begin{example}
	\label{exampleCrep}
	Let $E$ be a rank six vector bundle over a variety $X$ and consider a general global section $s$ of the globally generated vector bundle $\wedge^3 E$. Let $Y \subset \wedge^3 V_6$ be the subvariety of partially decomposable tensors introduced in Example \ref{partDec}, which is normal, has rational singularities and whose affine coordinate ring is Gorenstein. A locally free resolution of $\Odys$ was provided in Example \ref{partDec}, and the last term reads $(\det E^*)^5$. By Theorem \ref{GorensteinControl}, the canonical bundle of $\dys$ is
	\[
	K_{\dys} = (K_X\otimes (\det E)^5)_{|\dys},
	\]
	which coincides with what was computed in the previous paper \cite{BFMT} by means of a crepant Kempf collapsing of $Y$, see also Remark \ref{remarkoldpaper} below.
\end{example}

All the results proven so far can be put together in the following proposition, which we will use in order to construct interesting projective varieties in the next sections.

\begin{proposition}
	\label{operative}
	Let $V$ be a $G$-representation, $A$ its affine coordinate ring and $Y \subset V$ a $G$-invariant subvariety resolved by a Kempf collapsing. Assume that the ideal $I_Y$ is such that $A/I_Y$ is a Gorenstein ring, and let $V_\lambda$ be the one-dimensional representation appearing in the last term of a minimal $G$-equivariant free resolution of $A/I_Y$. Let $X$ be a variety, $\cE$ a $G$-principal bundle on $X$ and $s$ a general section of the globally generated vector bundle $\cE_V$. Then
	\begin{itemize}[leftmargin=2.6ex]
		\item $\codim_X \dys = \codim_V Y$;
		\item $\Sing(\dys)=D_{\Sing (Y)}(s)$ and $\codim \Sing (\dys) = \codim \Sing (Y)$;
		\item $\dys$ is Gorenstein and $\omega_{\dys}=(K_X\otimes \cE_{V_\lambda}^*)_{|\dys}$;
		\item if $Y$ is normal, then $\dys$ is normal;
		\item if $Y$ has rational singularities, then $\dys$ has canonical rational singularities.
	\end{itemize}
	\begin{proof}
		The first two statements and the fourth follow from Proposition \ref{codimNormality}. The third statement follows from Theorem \ref{GorensteinControl}. The last statement follows from \ref{codimNormality} since Gorenstein rational singularities are canonical, see e.g.\ \cite[Corollary 11.13]{Kollar97}.
	\end{proof}
\end{proposition}

\begin{remark}
	\label{remarkoldpaper}
	Recall that, under the hypotheses of Proposition \ref{propressing}, it is possible to construct a resolution of singularities $\theta' : \zero(\tilde{s})\rightarrow D_Y(s)$ as the zero locus of a vector bundle on what was called $\cE_{G/P}$.
	
	In \cite{BFMT}, we introduced crepant Kempf collapsings, i.e.\ Kempf collapsings $\cW \rightarrow Y$ such that $\det \cW = K_{G/P}$. If $Y$ admits such a resolution of singularities, and under the hypotheses of Proposition \ref{propressing}, it turns out \cite[Proposition 2.7, Proposition 2.8]{BFMT} that the canonical bundle of $\zero(\tilde{s})$ is the restriction of the pull-back of a line bundle on $X$. If $\dys$ is smooth or $Y$ has rational singularities, this implies that the canonical bundle of $\dys$ is the restriction of some line bundle on $X$, which was computed through the resolution of singularities by means of the adjunction formula. Such information was used in \cite{BFMT} to construct projective varieties with trivial or negative canonical bundle.
	
	Theorem \ref{GorensteinControl} shows that the same happens whenever the $G$-stable subvariety $Y \subset V$ has Gorenstein affine coordinate ring: the dualizing sheaf of $\dys$ is the restriction of a line bundle on $X$, which can be computed from an equivariant free resolution of the affine coordinate ring of $Y$.
	
	In Proposition \ref{crepantIsGorenstein}, we will show that the crepancy of a Kempf collapsing resolving $Y$ actually implies the Gorenstein property for its affine coordinate ring, at least when $Y$ has rational singularities. This means that Theorem \ref{GorensteinControl}, or Proposition \ref{operative}, can be read as generalizations of the methods in \cite{BFMT}.
\end{remark}

\subsection{Crepant Kempf collapsings and Gorenstein rings}
\label{crepKempfColl}

Recall that a Kempf collapsing of a $G$-stable subvariety $Y$ of a $G$-module $V$ is a resolution 
of singularities $\cW\rightarrow Y$ given by the total space of a vector bundle $\cW$ on some generalized flag manifold $G/P$, defined by a $P$-submodule of $V$.

\begin{definition}
	\label{defCrep}
	A Kempf collapsing $\cW\rightarrow Y$ is said to be \emph{crepant} if $\det \cW = K_{G/P}$.
\end{definition}

When a birational Kempf collapsing is crepant, the canonical sheaf $K_Y$ is trivial and the induced resolution $\mathbf{P}(\cW)\rightarrow \bar{Y}$ of the projectivization of $Y$ is crepant, see \cite[Proposition 2.7]{BFMT}.

\begin{proposition}
	\label{crepantIsGorenstein}
	Let $Y$ be a $G$-stable subvariety of a $G$-module $V$ resolved by a crepant Kempf collapsing. 
	Suppose moreover that $Y$ has rational singularities. Then the affine coordinate ring of $Y$ is Gorenstein.
	\begin{proof}
		Rational singularities are Cohen--Macaulay, hence the singularities of $Y$ as an affine variety are Cohen--Macaulay and its affine coordinate ring is Cohen--Macaulay as well. The conclusion follows since the canonical sheaf $K_Y$ is trivial by \cite[Proposition 2.7]{BFMT}.
	\end{proof}
\end{proposition}

\begin{remark}
	\label{GorAndCrep}
	If a given birational Kempf collapsing is not crepant, we cannot decide in general whether the closure of the orbit has Gorenstein affine coordinate ring or not (recall, by the way, that a particular $Y$ may very well admit some crepant Kempf collapsings and other non-crepant Kempf collapsings at the same time, see e.g.\ \cite[Remark 2.14]{BFMT}). Nonetheless, having a crepant Kempf collapsing is a sufficient condition which is in general easier to check.
	
	Conversely, a $G$-invariant subvariety $Y$ with rational singularities and Gorenstein affine coordinate ring may very well 
	admit no Kempf collapsing at all. Indeed, we have the following
	
	\begin{proposition}[{\cite[Proposition 2.7]{BFMT}}]
		Let $Y$ be normal and suppose that there exists a crepant birational Kempf collapsing $\cW \rightarrow Y$. Then the resolution of singularities $\PP({\cW}) \rightarrow \bar{Y}$ is crepant and $K_{\bar{Y}}=\cO_{\bar{Y}}(- \rank (\cW))$.
	\end{proposition}
	
	This gives a strong restriction on the dimension of the base $G/P$. For example, the cone $Y$ over $\Gr(2,5)$ has dimension seven, and since $\Gr(2,5)$ has index five,
	the base of a crepant Kempf collapsing resolving the singularities of $Y$ should be a surface. But there is no $GL_5$-homogeneous space of dimension two!
\end{remark}

\begin{remark}
	For an affine variety $Y$ which is a cone sometimes some confusion may arise when talking about the affine coordinate ring being Gorenstein or the projectivization of $Y$ being Gorenstein as a projective variety. We remark here that the former condition implies the latter, as can be seen by localizing a minimal graded resolution of the ideal of $Y$ inside the coordinate ring of $V$. However, the converse is not always true, as shown by any smooth projective variety having ideal $I$ with free resolution of length bigger than the codimension of $I$.
	
	In the rest of the paper, when no ambiguity arises, we will say that an affine $G$-invariant subvariety $Y$ is Gorenstein when its affine coordinate ring is Gorenstein.
\end{remark}

\subsection{Constructing new varieties}
\label{subsectConstructing}

Our main motivation for the study of ODL
is to construct new interesting varieties, together with their resolutions of singularities and locally free resolutions of their structure sheaves. The construction is very general and could potentially be used in a great variety of situations. 

In \cite{BFMT}, crepant Kempf collapsings resolving the affine subvarieties $Y$ were used to write down the canonical bundle of the ODL associated to $Y$ in terms only of the starting vector bundle and the ambient variety. This allowed us to to construct (almost) Fano varieties and Calabi--Yau varieties. However, crepant Kempf collapsings are in general not easy to find and, for a given $Y$, may very well not exist.

The study of the canonical bundle of an ODL that we made in the Gorenstein case, in particular Proposition \ref{operative}, leads to a more effective strategy to construct varieties with negative or trivial canonical bundle.  In the next sections we will focus on describing several cases of Gorenstein orbit closures, together with the data needed to control the canonical bundle of the corresponding ODL. In particular, we will look at parabolic orbits, as they provide many cases of orbit closures for which the Gorenstein property can be (and indeed has already been) studied in a systematic way.

Our strategy will be the following.

\begin{enumerate}[leftmargin=2.6ex]
	\item For an assigned Gorenstein orbit closure $Y$ inside a $G$-representation $V$, we will consider a $G$-principal bundle over a variety $X$ and construct $\cE_V$, the associated vector bundle on $X$. When $G=GL_n$, such a $G$-principal bundle is the frame bundle of a rank $n$ vector bundle on $X$, which will be our basic datum. For more complicated groups like $G=Spin_n, Sp_{2n}, E_6, E_7$, we will construct $G$-principal bundles, and the associated bundle $\cE_V$,  from vector bundles on $X$ with some additional structures, see Sections \ref{aReminder}, \ref{E7alpha7}, \ref{e8alpha8}.
	\item For a general section of a globally generated $\cE_V$, we will compute the canonical bundle $K_{\dys}$ of the corresponding ODL. This computation can be performed in different ways, depending on the information we have on the chosen $G$-invariant Gorenstein subvariety $Y$. The main piece of 
	information that is relevant here is the last term $A(-N)$  of a free resolution of $A/I_Y$. Here
	we omit the group action, which will be important later on. Note that the group $G$ acts by a character; sometimes it has none, but since $Y$ is always a cone there is a $\CC^*$-action that 
	preserves the whole construction.
	
	Let us first focus on $N$. When the minimal free resolution 
	of $A/I_Y$ is not known, we can guess the value of $N$ thanks to the following tricks:
	\begin{itemize}[leftmargin=2.6ex]
		\item let $\HS_Y(t)=p(t)/(1-t)^{\dim Y}$ be the Hilbert series of $Y$. For any $i$, the coefficient in degree $i$ of $p(t)\cdot (1-t)^{\codim Y}$ expresses the alternating sum of the Betti numbers $\beta^A_{i-j,j}(A/I_Y)$. This implies that $$N=\codim Y+ \deg (p(t));$$
		\item in the special case where $Y$ is the cone over a generalized Grassmannian $G/P$  for some maximal parabolic subgroup $P$, we can compute by means of representation theory the canonical bundle of $G/P$ and recover $N$ a posteriori, see also Section \ref{sec.others}. 
	\end{itemize}
	Then we turn to the relative setting of a $G$-principal bundle over a variety $X$, and an ODL
	$D_Y(s)$. By Proposition \ref{operative} we know that we can write the canonical bundle of $\dys$ as the restriction of a suitable line bundle over the ambient variety $X$.
	\begin{itemize}[leftmargin=2.6ex]
		\item If we know a crepant resolution of $Y$, then we have an induced rational resolution of singularities of $\dys$ given by the zero locus of a section of a suitable vector bundle. The adjunction formula yields the canonical bundle of both  $\dys$ and its resolution;
		\item if no crepant resolution is known, we can use the fact that the canonical bundle $K_{D_Y(s)}$  is the restriction to $\dys$ of $K_X \otimes L$, $L^*$ being the last term in the induced locally free resolution of $\cO_{\dys}$ as an $\cO_X$-module. This minimal free resolution is a relative version of the 
		minimal free $A$-resolution of the ideal of $A/I_Y$, whose last term is $A(-N)$. This implies
		that $L$ is a factor of $\cE_V^{N}$ (we recall the notation $\cE_V^{N}=(\cE_V)^{\otimes N}$); moreover, the $\CC^*$-action that preserves 
		the cone $Y$ defines, in the relative setting, a line bundle $M$ on $X$, and since our 
		whole construction is preserved by this $\CC^*$-action, $L$ must be a power of 
		$M$. But the determinant of $\cE_V$, for the same reason, must also be a power of $M$, thus we can conclude that $L$ must be a (rational) power of $\det(\cE_V)$. The two pieces of information:
		\begin{enumerate}
			\item $K_{D_Y(s)}$ is the restriction of $K_X \otimes L$, $L$ being a factor of $\cE_V^{N}$,
			\item $L$ is a (rational) power of $\det(\cE_V)$,
		\end{enumerate}
		will always be enough to determine $K_{D_Y(s)}$ completely. 
	\end{itemize}
	\item With all the previous data, we know the codimension of $\dys$ inside $X$, (a bound for) the codimension of its singularities and its canonical bundle in terms of $\cE_V$. We can then look for particular choices of $X$ and $\cE_V$ in order to construct interesting varieties. For the sake of producing explicit examples and showing the effectiveness of our techniques, we will give for most of the cases considered a few examples of varieties with trivial canonical bundle, focusing in particular on four-dimensional ones.
\end{enumerate}

\section{Gorenstein parabolic orbits of classical type}

\label{GorParOrb}

Motivated by the previous section, and in particular by Theorem \ref{GorensteinControl} and Proposition \ref{operative}, we are led to look for Gorenstein $G$-invariant subvarieties $Y$ inside $G$-representations. These provide a source of cases which can be used to produce ODL whose canonical bundle can be controlled. 

An interesting source of $G$-invariant subvarieties is provided by the closure of parabolic orbits. The last two sections of this paper are mostly devoted to providing a list of Gorenstein parabolic orbit closures, together with the additional data we need to control the geometry of the corresponding ODL. After a reminder on parabolic orbits, in this section we will focus on those that can be 
constructed from the classical simple Lie algebras. The next section will be devoted to the exceptional
Lie algebras and their associated Gorenstein orbit closures. 

\subsection{Parabolic orbits}
\label{parOrb}
We recall here some basic facts about parabolic orbits. For more background we refer to \cite[\textsection 2.3]{BFMT} and the references therein.

Let $\fg$ be a simple Lie algebra. Suppose that a Cartan subalgebra $\fh$ has been chosen and 
consider the  root space decomposition
\[
\fg=\fh\oplus\bigoplus_{\alpha\in\Phi}\fg_\alpha.
\]
Suppose that a set $\Delta$ of simple roots has been fixed. For a given simple root 
$\alpha_i$, consider the linear form $\ell$ on the root lattice such that 
$\ell(\alpha_i)=1$ and $\ell(\alpha_j)=0$ for $j\ne i$. Then 
\[
\fg_k=\bigoplus_{\ell(\alpha)=k}\fg_\alpha\oplus\delta_{k,0}\fh
\]
is a $\ZZ$-grading of $\fg$. In particular $\fg_0$ is a Lie subalgebra of $\fg$, in fact 
a reductive subalgebra whose semisimple part has a Dynkin diagram deduced from that of $\fg$ just
by suppressing the node corresponding to $\alpha_i$. Moreover, each $\fg_k$ is a $\fg_0$-module,
which turns out to be irreducible. We will concentrate on $\fg_1$, whose lowest weight is $\alpha_i$ 
and which is therefore easy to identify. 

Let $G_0$ be the subgroup of $G=\Aut(\fg)$ with Lie algebra $\fg_0$. A {\it parabolic orbit} is a $G_0$-orbit in $\fg_1$, obtained from some $\ZZ$-grading 
of some simple Lie algebra $\fg$ associated to a simple root $\alpha_i$. 

In view of constructing ODL, parabolic orbit closures are an interesting source of examples for the following reasons.

\begin{itemize}[leftmargin=2.6ex]
	\item By \cite[Lemma 1.3]{Kac80}, there 
	are only finitely many $G_0$-orbits in $\fg_1$. Hence, a case by case study is feasible.
	\item A parabolic orbit closure admits a resolution of singularities given by a Kempf collapsing; this should be taken with a caveat, as for the exceptional group $E_8$ parabolic orbits remain a bit mysterious. By Section \ref{ODLandRes}, this means that an ODL constructed by considering a parabolic orbit closures admits a resolution of singularities which is the relative version of a Kempf collapsing over the ambient variety.
	\item Much information about the equivariant free resolution of the ideal of a parabolic orbit closure (and therefore about the orbit closure being Gorenstein) can be deduced by using Weyman's techniques. In particular, the free resolutions of some parabolic orbit closures associated to $\ZZ$-gradings of classical Lie algebras have been computed in \cite{Weyman2003}; the series of papers \cite{KW12, KW13, KWE8} deals with the case of exceptional Lie algebras.
\end{itemize}

In Section \ref{sec.matrices} we will examine the parabolic orbits in spaces of matrices, for which a complete classification of Gorenstein orbit closures is provided. In Section \ref{sec.mixed} we deal with mixed parabolic orbits; finally, in Section \ref{sec.others}, we study the case of cones over generalized Grassmannians.

\subsection{Orbits in spaces of matrices}
\label{sec.matrices}

In this section we collect some classical facts about orbits in spaces of matrices, which can be mostly deduced from the results in \cite{Weyman2003}.

\subsubsection{General matrices}

The case of general matrices is well-known, but we briefly discuss it for completeness. It has already been introduced in Example \ref{classicalDL}, where we exhibited the Eagon--Northcott complex resolving the ideal of matrices of corank at least one. Actually, in \cite{Weyman2003} the equivariant resolutions of all determinantal varieties $Y_r$ of (symmetric, skew-symmetric) matrices of rank at most $r$ are computed. From the complexes, it is easy to check
that determinantal orbit closures are Gorenstein only for square matrices: 

\begin{proposition}[{\cite[(6.1.5)]{Weyman2003}}]
	\label{classicaldet}
	Let $Y_r$ be the determinantal variety of matrices of rank at most $r$ inside $V_e^*\otimes V_f$. Then $Y_r$ has rational singularities, and it is Gorenstein if and only if $e=f$. In this case, the last term of the resolution of $\cO_{Y_r}$ is 
	\[
	F_c= (\det V_e)^{e-r}\otimes (\det V_f^*)^{e-r}\otimes A(-e(e-r))
	\]
\end{proposition}

Moreover, the variety $Y_r$ is singular along $Y_{r-1}$, i.e., if $e=f$, in codimension $2e-2r+1$. In the relative case, by applying Proposition \ref{operative}, we get:

\begin{proposition}
	Let $Y_r$ be as in Proposition \ref{classicaldet}. Let $X$ be a variety and $E$, $F$ two vector bundles of the same rank $e$ on $X$. Suppose that $E^*\otimes F$ is globally generated, and $s$ is a general section of this bundle. 
	Then the ODL $D_{Y_r}(s)$ has canonical rational singularities and
	\[
	\codim_X D_{Y_r}(s)=(e-r)^2, \qquad \codim_{D_{Y_r}} \Sing(D_{Y_r}(s))=2e-2r+1,
	\]
	\[
	K_{D_{Y_r}(s)}=(K_X\otimes (\det E^*)^{e-r}\otimes (\det F)^{e-r} )_{|D_{Y_r}(s)}.
	\]
\end{proposition}

\subsubsection{Skew-symmetric matrices}

Another class of examples of ODL that is already present in the literature is Pfaffian varieties, i.e.\ the degeneracy loci of skew-symmetric morphisms between vector bundles. Again, we consider a more general situation: we denote by 
$M^a$ the space of skew-symmetric $(e\times e)$-matrices. For $r$ even, we consider the determinantal variety $Y_r^a$ of matrices of rank at most $r$ inside $M^a$.

\begin{proposition}[{\cite[(6.4.1)]{Weyman2003}}]
	\label{skewsymmm}
	Let $Y_r^a$ be the determinantal variety of matrices of rank at most $r$ inside $M^a\cong \wedge^2 V_e$. Then $Y_r^a$ has rational singularities, it is Gorenstein, and the last term of the minimal resolution of $\cO_{Y_r^a}$ is 
	\[
	F_c= (\det V_e^*)^{e-r-1}\otimes A(-e(e-r-1)/2).
	\]
\end{proposition}

\begin{proposition}
	Let ${Y_r^a}$ be as in Proposition \ref{skewsymmm}. Let $X$ be a variety and $E$ a vector bundle of rank $e$ on $X$. Suppose that $\wedge^2 E$ is globally generated, and $s$ is a general section of this bundle. 
	The ODL $D_{Y_r^a}(s)$ has canonical rational singularities and satisfies:
	\[
	\codim_X D_{Y_r^a}(s)=(e-r)(e-r-1), \qquad \codim_{D_{Y_r^a}} \Sing(D_{Y_r^a}(s))=4e-4r+2,
	\]
	\[
	K_{D_{Y_r^a}(s)}=(K_X\otimes (\det E)^{e-r-1} )_{|D_{Y_r^a}(s)}.
	\]
\end{proposition}

\subsubsection{Symmetric matrices}

Let us denote by $M^s$ the space of symmetric $(e\times e)$-matrices. We consider the determinantal variety $Y^s=Y_r^s$ of matrices of rank at most $r$ inside $M^s$.

\begin{proposition}[{\cite[(6.3.1)]{Weyman2003}}]
	\label{symmm}
	Let $Y^s_r$ be the determinantal variety of matrices of rank at most $r$ inside $M^s\cong Sym^2 V_e$. Then $Y^s_r$ has rational singularities, and it is Gorenstein if and only if $e-r$ is odd. With this hypothesis, the last term of the minimal resolution of $\cO_{Y^s_r}$ is 
	\[
	F_c= (\det V_e^*)^{e-r+1}\otimes A(-e(e-r+1)/2).
	\]
\end{proposition}

\begin{proposition}
	Let $Y^s_r$ be as in Proposition \ref{symmm}. Let $X$ be a variety and $E$ a vector bundle of rank $e$ on $X$. Suppose that $Sym^2 E$ is globally generated, and $s$ is a general section of this bundle. 
	If $e-r$ is odd, then the ODL $D_{Y^s_r}(s)$ has canonical rational singularities and satisfies:
	\[
	\codim_X D_{Y^s_r}(s)=(e-r)(e-r+1), \qquad \codim_{D_{Y_r^s}} \Sing(D_{Y_r^s}(s))=2e-2r+2,
	\]
	\[
	K_{D_{Y^s_r}(s)}=(K_X\otimes  (\det E)^{e-r+1})_{|D_{Y^s_r}(s)}.
	\]
\end{proposition}

The subvarieties of square matrices with rank bounded from above are the parabolic orbits coming from $(A_{2n-1}, \alpha_n)$, i.e.\ from the $\mathbb{Z}$-grading of $\fsl_{2n}$ associated to the simple root $\alpha_n$; analogously, the determinantal subvarieties of symmetric (respectively, skew-symmetric) matrices come from $(C_n,\alpha_n)$ (respectively, $(D_n,\alpha_n)$).

We can, in a few other classical cases, exhibit Gorenstein parabolic orbit closures, by constructing a crepant resolution of singularities and applying Proposition \ref{crepantIsGorenstein}. This will be the subject of the next subsection. A more thorough study would be required to decide whether there exist other Gorenstein orbit closures than those we are about to list.

For the exceptional cases, in the series of papers \cite{KW12,KW13,KWE8} Gorenstein parabolic orbit closures are completely determined. In Section \ref{sec.exceptionals} we will use this information to study the ODL associated to them.

\subsection{Mixed parabolic orbits}
\label{sec.mixed}

Let us consider the parabolic representations given by $(D_n,\alpha_k)$ with $2\leq k \leq n-3$, $(B_n,\alpha_k)$ with $2\leq k \leq n-2$, and $(C_n,\alpha_k)$ with $2\leq k \leq n-2$.
These correspond to the action on $\Hom(V_2, V_1)$ of $G_0=GL(V_1)\times SO(V_2)$ or 
$GL(V_1)\times Sp(V_2)$, where $V_1$ has dimension $d_1=k$, and $V_2$ has dimension 
$d_2=2(n-k)$ (resp.\ $2(n-k)+1$, $2(n-k)$) endowed with a symmetric (resp.\ symmetric, skew-symmetric) non-degenerate bilinear form $q$. 

The orbits are defined by two integers $r$, $d$:
\[
O_{r,d}=\Bigg\{\begin{array}{c}
\psi \in \Hom(V_2, V_1) \mbox{ s.t.\ } \rank(\psi)= r\\
\mbox{ and } \rank (q_{|\ker \psi})= d_2-r-d
\end{array}\Bigg\},
\]
where $q_{|\ker \psi}$ denotes the restriction of the bilinear form $q$ to the kernel of $\psi$. The condition $\rank (q_{|\ker \psi})= d_2-r-d$ means that the kernel of $q_{|\ker \psi}$ is $d$-dimensional. This 
kernel is $(\ker \psi)\cap (\ker \psi)^\perp$, where $(\ker \psi)^\perp=\{z \in \ker \psi : q(z,z')=0 \;\forall z' \in \ker \psi\}$; this means that there is an isotropic space $L$ of dimension $d$ 
such that $L\subset \ker \psi\subset L^{\perp}$. The orthogonal $L^{\perp}$ has dimension $d_2-d$ and its image $I$ by $\psi$
has dimension $r-d$. 

This yields a natural Kempf collapsing resolving the singularities 
of $\overline{O}_{r,d}$: the base is the product $\SGr(d,V_2)\times F(r-d,r,V_1)$ of triples $(L,I,J)$ with $L$ isotropic of dimension 
$d$ in $V_2$ (and $\SGr$ is $\OGr$ or $\IGr$),  
and $I\subset J\subset V_1$ of dimensions $r-d$ and $r$; the vector bundle $\cW$ to be considered on that base 
has fiber over $(L,I,J)$ defined by 
\[
\cW_{(L,I,J)}=\{ \psi\in \Hom (V_2,V_1) \mbox{ s.t.\ } \psi(L)=0, \;  \psi(L^{\perp})\subset I, \; \psi(V_2)\subset J\}.
\]
Unfortunately, a straightforward computation shows that this Kempf collapsing is never crepant. 

However, when $d=d_2-r$, it is possible to construct a different Kempf collapsing of $\overline{O}_{r,d}$. Indeed, in this case, the points $\psi$ in the orbit satisfy $\rank (q_{|\ker \psi})=0$, which means that $\ker \psi$ is isotropic with respect to $q$. A resolution of such orbit closure is given by the total space of the bundle $\HOM(\cQ,V_1)$ over $\SGr(d_2-r,V_2)$. This Kempf collapsing is crepant when $d_1$ is equal to the index of $\SGr(d_2-r,V_2)$, which is $r-\epsilon$ with $\epsilon=1$ in the symmetric case and $\epsilon=-1$ in the skew-symmetric case.
Since of course $r\le d_1$, we must restrict to the latter case, and we get a crepant Kempf collapsing of the closure of 
$O_{d_1-1, d_2-d_1+1}$, where $d_2\leq 2d_1-2$.

\begin{proposition}
	Let $Y=\overline{O}_{d_1-1, d_2-d_1+1}\subset \Hom(V_2,V_1)$, where $V_1$ has dimension $d_1$, and $V_2$ is symplectic of 
	dimension $d_2$ with $d_1-1 \leq d_2\leq 2d_1-2$. Then: 
	\begin{enumerate}
		\item $Y$ is normal, Gorenstein, and has rational singularities;
		\item the codimension in $Y$ of its singular locus is equal to three. 
	\end{enumerate}
\end{proposition}
\begin{proof}
	Let us begin with the second statement. The complement $\partial Y$ of the open orbit in $Y$ is the image of the locus 
	inside the bundle $\HOM(\cQ,V_1)$ where the rank is not maximal. This locus has codimension two. Moreover a
	general point of its image is a morphism $\psi\in \Hom(V_2,V_1)$ whose rank is $r-1$, and such that $q_{|\ker \psi}$
	has rank two. Then $\ker \psi$ contains a one dimensional family of isotropic hyperplanes. This proves that 
	$\partial Y$ has codimension three in $Y$, and it must be contained in its singular locus since it is in 
	the image of the exceptional locus of the Kempf collapsing, which does not contract any divisor. 
	
	In order to prove the first statement, we cannot use Theorem \ref{KempfInventiones} since the homogeneous vector bundle $\HOM(\cQ,V_1)$ over $\mathbb{G}:=\IGr(d_2-d_1+1,V_2)$ is not completely reducible. But we can apply Weyman's geometric technique: by \cite[(5.1.3)]{Weyman2003}, it is sufficient to show that for all $j > 0$
	\begin{equation}
		\label{WeymansTechniques}
		\HHH^j(\mathbb{G},Sym(\HOM(V_1,\cQ)))=0 \qquad \mbox{and} \qquad \HHH^j(\mathbb{G},\wedge^j \HOM(V_1,\cU))=0
	\end{equation}
	to deduce that $Y$ is normal and has rational singularities, which by Proposition \ref{crepantIsGorenstein} implies the Gorenstein property. 
	
	Each $Sym^k(\HOM(V_1,\cQ))$ can be resolved by a complex whose $j$-th term is $Sym^{k-j}(V_1^* \otimes V_2)\otimes\wedge^j \HOM(V_1,\cU)$ for $j\le k$. 
	By the Cauchy formula, $\wedge^j \HOM(V_1,\cU)$ is a sum (with multiplicities) of Schur powers $S_\lambda \cU$, 
	taken over partitions $\lambda$ of size $|\lambda|=j$, with $\lambda_1\le d_1$. Therefore, in order to prove the vanishing conditions (\ref{WeymansTechniques}), we just need to check that $\HHH^q(S_\lambda \cU)=0$ for $q\geq|\lambda|>0$ and any such partition $\lambda$.

	To compute the cohomology 
	of $S_\lambda \cU$ over $\mathbb{G}$, we use the Borel--Weil--Bott Theorem, according
	to which there is at most one non-zero cohomology group, occurring in degree equal to the number of  {\it inversions}
	of the sequence $\lambda^*+\rho$. Here we work inside the weight lattice of the symplectic group, with its usual
	basis; then $\rho=(d,d-1,\ldots ,1)$ if $d_2=2d$, and $\lambda^*=(-\lambda_s,\ldots , -\lambda_1, 0,\ldots , 0)$
	if $s=d_2-r$, the string of zeros having length $t=d-s$. Thus 
	$$\lambda^*+\rho = (d-\lambda_s,\ldots , t+1-\lambda_1, t,\ldots , 1).$$
	For the corresponding cohomology groups not to vanish simultaneously, there must be no positive root $\alpha$ such that 
	$(\lambda^*+\rho, \alpha)=0$. Inversions correspond to those positive roots such that  $(\lambda^*+\rho, \alpha)<0$.
	Recall that the positive roots are the $\epsilon_i-\epsilon_j$ for $i>j$, and the  $\epsilon_i+\epsilon_j$ for $i\ge j$. 
	This gives rise to two different types of inversions. 
	
	The number of inversions of the first type is $q_1=\ell t$, if $\ell$ is such that $t+\ell-\lambda_\ell<-t$ and 
	$t+\ell+1-\lambda_{\ell+1}>t$. This means that $\lambda_\ell>2t+\ell$ and $\lambda_{\ell+1}\le\ell$. 
	The number of inversions of the second type is then equal to $q_2=\ell t+\binom{\ell+1}{2} +i_{1}+\cdots +i_\ell$, where 
	$$i_k=\# \{j>\ell, (t+j-\lambda_j)+(t+k-\lambda_k)<0\}.$$
	Note that, necessarily, $\lambda_k\ge 2t+k+1+i_k$. We deduce that 
	$$|\lambda|\ge (2t+1) \ell + \binom{\ell+1}{2}+i_{1}+\cdots +i_\ell=\ell+q_1+q_2=\ell+q.$$
	Moreover, if  $\ell=0$, then $(\lambda^*+\rho, \alpha)$ consists 
	in positive integers only, and therefore there is no inversion, so $q=0$, and we are done.
\end{proof}

In the relative setting, we deduce the following statement:

\begin{proposition}
	Let $X$ be a variety with a line bundle $L$, and two vector bundles $E_1$, $E_2$ of ranks $d_1$, $d_2$ with 
	$d_1-1 \leq d_2\leq 2d_1-2$. Suppose that $E_2$ is endowed with an everywhere non-degenerate skew-symmetric form 
	with values in $L$. If $\HOM(E_2, E_1)$ is globally generated and $s$ is a general section, then $\dys$ has canonical rational singularities and
	\[
	\codim_X D_{Y}(s)=\binom{d_2-d_1+2}{2},
	\qquad 
	\codim_{D_{Y}(s)}\Sing (D_{Y}(s)) =3,
	\]
	\[
	K_{D_Y(s)}=(K_X\otimes (\det E_1 \otimes \det E_2^*)^{d_2-d_1+1} \otimes L^{\binom{d_2-d_1+1}{2}} )_{|D_Y(s)}.
	\]
\end{proposition}

For example, if $d_1=3$ and $d_2=4$, one could take $E_2$ and $L$ to be trivial vector bundles; then we obtain a fourfold with trivial canonical bundle (singular along a curve) for each of the following choices for $X$ and $E_1$: 
\begin{equation*}
	\begin{split}
		X=\Gr(2,6)\cap Q, & \qquad  E_1=\cU^*\oplus \cO(1), \\
		X=\Gr(2,6)\cap Q, & \qquad  E_1=\cO_X\oplus 2\cO(1), \\
		X=\Gr(3,6)\cap H_1\cap H_2, & \qquad  E_1=\cO_X\oplus 2\cO(1), \\
		X=\Gr(3,6)\cap Q_1\cap Q_2, & \qquad  E_1=\cU^*, \\
		X=\Gr(3,6)\cap C\cap H, & \qquad  E_1=\cU^*.
	\end{split}
\end{equation*}
Here we denoted by $H$ (and $H_1, H_2$) a hyperplane, by $Q$ (and $Q_1, Q_2$) a quadric, by $C$ a cubic. All these loci admit a desingularization with trivial canonical bundle and Euler characteristic equal to two.

\subsubsection{$(D_4,\alpha_{2})$}

Starting from $(D_n,\alpha_{n-2})$ (the root corresponding to the triple node of the Dynkin diagram),
we get the representation of $G=GL(V_1)\times GL(V_2)\times GL(V_3)$ in   
$V_1\otimes V_2\otimes V_3$, where  
$\dim(V_1)=\dim(V_2)=2$, $\dim(V_3)=n-2$.
We identify the representation with $\Hom(V_1^*\otimes V_2^*, V_3)$. It is then easy to see 
that two morphisms are in the same orbit if and only if they have same rank and same dimension of the intersection of the kernel (if non-trivial) with the Segre variety 
$\PP(V_1^*)\times \PP(V_2^*)\subset \PP(V_1^*\otimes V_2^*)$. If $n-2\geq 4$ there are $10$ different orbits (if $n-2=1,2,3$ there are respectively $3,7,9$ orbits). Each of them is easy to desingularize,
but we could find a crepant resolution only for $n=4$, for the minimal non-trivial orbit $Y$ given by the cone $Y$ over $\PP(V_1)\times \PP(V_2)\times\PP(V_3)=\PP^1 \times\PP^1\times\PP^1$. The resolution is by the total space of 
$\cO(-1,-1)\otimes V_3$ over $\PP(V_1)\times \PP(V_2)$. By Theorem \ref{KempfInventiones}, $Y$ has rational singularities, and being a cone over a smooth variety it is singular only at the origin, i.e. in codimension $4$. We deduce:

\begin{proposition}
	Let $X$ be a variety with three rank two vector bundles $E_1$, $E_2$, $E_3$. Suppose that $E_1\otimes E_2\otimes E_3$ is globally generated, and let $s$ be a general section. 
	Then $\dys$ has canonical rational singularities, and
	\[
	\codim_X D_{Y}(s)=4, \qquad 
	\codim_{D_{Y}(s)}\Sing (D_{Y}(s))=4, 
	\]
	\[
	K_{D_Y(s)}=(K_X\otimes (\det E_1)^3\otimes (\det E_2)^3 \otimes (\det E_3)^3 )_{|D_Y(s)}.
	\]
\end{proposition}

As an application, one could take $X=\Gr(2,6)$ and $E_1=E_2=\cU^*$, $E_3=2\cO_X$, where $\cU^*$ is the dual tautological
bundle. One would get a family of fourfolds with trivial canonical bundle, and $c_4(\cU^*\otimes \cU^*)^2=32$ isolated singularities, whose resolution is of Calabi--Yau type. 

\subsection{Cones over generalized Grassmannians}
\label{sec.others}

Apart from parabolic spaces, many other closed $G$-stable subvarieties of a $G$-module can be considered in order to construct ODL. 
For instance, quiver representations are considered in \cite{Benedettiorbit}; some of their orbit closures admit crepant resolutions
and are Gorenstein.

The cone $Y$ over a generalized Grassmannian $G/P$ , where $P$ is a maximal parabolic subgroup of the 
simple Lie group $G$ associated to a simple root $\alpha_i$ and $G/P$ is embedded in $\mathbf{P}(V_{\omega_i})$, also appears to be Gorenstein in many cases. This is notably true when
$G/P$ is an ordinary Grassmannian \cite[(7.3.6)]{Weyman2003}, and also (by a case by 
case inspection) for all the extremal 
parabolic cases  
(i.e.\ $Y$ is the closure of the minimal non-trivial parabolic orbit coming from $(\Delta,\alpha)$, when $\alpha$ corresponds to an end of the simple Dynkin diagram $\Delta$).

The above observations lead to the following
\begin{question}
	\label{thequestion}
	Is it true that the cone over a generalized Grassmannian $G/P$, embedded in the projectivization $\mathbf{P}(V)$ of the corresponding fundamental representation $V$, is always Gorenstein?
\end{question}

\begin{example}
	Let $(\Delta,\alpha)=(C_n,\alpha_n)$. The minimal non-trivial parabolic orbit closure is the cone over $\PP(V_n)$ embedded in $\PP(Sym^2 V_n)$ via the Veronese embedding, and it is not Gorenstein. However, even though $\PP(V_n)$ is a generalized Grassmannian, this case does not give a counterexample to Question \ref{thequestion}, as it appears here as a parabolic orbit inside the projectivization of a representation which is not the corresponding fundamental representation.
	
	The case $(G_2, \alpha_2)$ shares the same behavior, as it gives rise to the cone over a rational normal curve in $\PP^3$.
\end{example}

We remark that if the cone over $G/P$ is Gorenstein, and if the last term of its minimal resolution is $A(-N)$, 
then the canonical bundle of $G/P$ is $\cO_{G/P}(-\dim V+N)$. This gives
\begin{equation}
	\label{Niota}
	N=\dim V-\iota_{G/P}, 
\end{equation}
where $\iota_{G/P}$ denotes the index of $G/P$. 

\medskip For cones over Grassmannians, which are always singular at the origin ($\codim_{Y}\Sing (Y)=\dim(Y)$) and have rational singularities by \cite[(7.1.2)]{Weyman2003}, we get the following statement: 

\begin{proposition}
	\label{ordGrassODL}
	Let $E$ be a rank $e$ vector bundle on a variety $X$. Suppose that $\wedge^k E$ is globally generated, and let $s$ be a general section. For $Y$ the cone over the Grassmannian $\Gr(k,e)\subset \PP(\wedge^k \mathbb{C}^e)$, we have that $\dys$ has canonical rational singularities and
	\[
	\codim_X D_{Y}(s)={\binom{e}{k}}-k(e-k)-1, \quad \codim_{D_{Y}(s)}\Sing (D_{Y}(s))=k(e-k)+1, 
	\]
	\[
	K_{D_Y(s)}=\left(K_X\otimes \det(E)^{{\binom{e-1}{k-1}}-k} \right)_{|D_{Y}(s)}.
	\]
\end{proposition}

\section{Gorenstein parabolic orbits of exceptional type}
\label{sec.exceptionals}

In the series of papers \cite{KW12,KW13,KWE8}, Gorenstein parabolic orbit closures coming from exceptional Lie groups are thoroughly investigated. In this section, building on these results, we study the ODL associated to these orbit closures. Each parabolic representation is studied in a different subsection according to Table \ref{tableGorenstein}, in which we listed for completeness all Gorenstein orbit closures which are not hypersurfaces. They all turn out to have rational singularities.

{
	\begin{table}[h!bt]
		\begin{center}
			\caption{Gorenstein orbits in parabolic representations}
			\label{tableGorenstein}
			\resizebox{\textwidth}{!}
			{
				\begin{tabular}{cccccc} \toprule
					$(D,\alpha)$ & $G$ & $V$ & $\begin{array}{c}\mbox{Cones}\\\mbox{over } G/P\end{array}$ & $\begin{array}{c}\mbox{Other}\\\mbox{orbits}\end{array}$ & Section
					\\	\midrule
					$(E_6,\alpha_1)$ & $Spin_{10}$ & $V_{\omega_5}$ & $Y_5$ & & \ref{E6alpha1}
					\\	\midrule
					$(E_6,\alpha_2)$ & $GL_6$ & $\wedge^3 \mathbb{C}^6$ &  $Y_{10}$ &$Y_5$ &\ref{E6alpha2}
					\\	\midrule
					$(E_7,\alpha_1)$ & $Spin_{12}$ & $V_{\omega_6}$ &  $ Y_{16}$ & $Y_7$ &\ref{E7alpha1}
					\\	\midrule
					$(E_7,\alpha_3)$ & $GL_2 \times GL_6$ & $\mathbb{C}^2 \otimes \wedge^2 \mathbb{C}^6$ & &$Y_4, Y_{10}$  &\ref{E7alpha3}
					\\	\midrule
					$(E_7,\alpha_6)$ & $GL_2 \times Spin_{10}$ & $\mathbb{C}^2 \otimes V_{\omega_5}$ & &$Y_8, Y_{13}$  & \ref{E7alpha6} 
					\\	\midrule
					$(E_7,\alpha_7)$ & $E_6$ & $V_{\omega_6}$ &  $Y_{10}$ && \ref{E7alpha7} 
					\\	\midrule
					$(E_8,\alpha_1)$ & $Spin_{14}$ & $V_{\omega_7}$ & $Y_{42}$ &$Y_5, Y_{10},Y_{14},Y_{29}$ &  \ref{E8alpha1} 
					\\	\midrule
					$(E_8,\alpha_2)$ & $GL_8$ & $\wedge^3 \mathbb{C}^8$ & $ Y_{40}$ &$Y_4, Y_{25}$ &  \ref{E8alpha2} 
					\\	\midrule
					$(E_8,\alpha_6)$ & $GL_3 \times Spin_{10}$ & $\mathbb{C}^3 \otimes V_{\omega_5}$ && $ Y_9$  & \ref{E8alpha6} 
					\\	\midrule
					$(E_8,\alpha_7)$ & $GL_2 \times E_6$& $\mathbb{C}^2 \otimes V_{\omega_1}$ & &$Y_4, Y_7, Y_{25}$  & \ref{e8alpha7} 
					\\	\midrule
					$(E_8,\alpha_8)$ & $E_7$ & $V_{\omega_7}$ & $ Y_{28}$ &$Y_{11}$ &  \ref{e8alpha8} 
					\\	\midrule
					$(F_4,\alpha_1)$ & $Sp_6$ & $V_{\omega_3}$ &  $Y_{7}$ && \ref{f4alpha1} 
					\\ \bottomrule
			\end{tabular}}
		\end{center}
	\end{table}
}

\subsection{A reminder on spin structures}
\label{aReminder}

Among the most interesting parabolic representations coming from the exceptional Lie algebras,
one can find the spin modules of $Spin_n$ for $n\le 14$. In this section, we collect a few 
classical facts about spin modules that will be needed later. 

We start with a vector space $V$ of dimension $n$, endowed with a non-degenerate quadratic form. 
We suppose that $n=2e$ is even, and that $V$ has been split into the direct sum of two maximal 
isotropic subspaces $E$ and $F$; the quadratic form then restricts to a duality between $E$ and $F$. We leave to the reader the modifications that are required to treat the case where $n$ is odd.

The half-spin representations can be defined as 
\[
S_+(E)=\bigoplus_{k\ge 0}\wedge^{2k}E, \qquad S_-(E)=\bigoplus_{k\ge 0}\wedge^{2k+1}E.
\]
Recall that the wedge product by elements of $E$, and the contraction by elements of $F\simeq E^*$, 
allow to define an action of $V$ on the exterior algebra of $E$, that extends to the Clifford 
algebra of $V$. This restricts to maps
\begin{equation}
	\label{refNaturalMaps}
	V\otimes S_{\pm}(E)\rightarrow S_\mp(E),
\end{equation}
and  to an 
action of $Spin(V)$ on the half-spin modules. Moreover $S_+(E)$ and $S_-(E)$ are self-dual when the 
dimension of $E$ is even, and dual to one another when the dimension of $E$ is odd.

A useful observation is that on the orthogonal Grassmannian $\OGr(k,V)$, where $k\leq e-2$, there exist two 
homogeneous vector bundles  $\cT_{+}$, $\cT_{-}$ of rank $2^{e-k-1}$, called half-spin bundles:  they are subbundles 
respectively of the trivial vector bundles with fiber $S_+(E)$ and $S_-(E)$. They are constructed as follows: suppose that $k$ is even (the odd case is similar). An isotropic space $P\in \OGr(k,V)$ defines (up to scalar) morphisms 
$\eta_{+}(P):S_+(E)\to S_+(E)$ and $\eta_-(P): S_-(E)\to S_-(E)$, induced by the action of the Clifford algebra of 
$V$ on $S_+(E)$ and $S_-(E)$. Then
\[
(\cT_+)_P=\im \eta_+(P)\subset S_+(E), \quad (\cT_-)_P=\im \eta_-(P)\subset S_-(E).
\]
Over $\OGr(e,V)$, one of the half-spin bundles is a line bundle; as is well-known, this line bundle is the square root 
of the Pl\"ucker line bundle restricted from $\Gr(e,V)$: we denote it $\cT_+=\cO(-\frac{1}{2})$.

In general, for $k\leq e-2$, we have the following formula:
\[
\det(\cT_{\pm})=\cO(-2^{e-k-2}).
\]
This can be derived from the fact that, if we denote by $\pi:\OF(k,e,V)\to \OGr(k,V)$ the natural projection from the orthogonal flag variety of isotropic subspaces of dimension $k$ and $e$, and $\cU_k\subset\cU_e\subset V\otimes \cO_{\OF(k,e,V)}$ the tautological bundles over $\OF(k,e,V)$ of respective ranks $k$ and $e$, we have that $\pi^*\cT_{+}$ admits a filtration whose terms are
\[
\cO(\frac{1}{2})\otimes \det(\cU_k)\otimes \wedge^{2h}(\cU_e / \cU_k) \mbox{ for all } h\geq 0,
\]
for $k$ even (and similarly for $\pi^*\cT_{-}$ and $k$ odd).

\medskip
For the relative version of these constructions, we need a split quadratic vector bundle $V$
on a variety $X$. That means that $V$ is endowed with an everywhere non-degenerate quadratic
form, that we allow to take values in a line bundle $L$. Moreover, we suppose  that $V$ is split
into the sum of two isotropic subbundles, and we write this decomposition as 
$V=E\oplus (E^*\otimes L)$. 
The associated Lie algebra bundle is 
\[
\mathfrak{so}(V)= \wedge^2 V\otimes L^*= (\wedge^2 E\otimes L^*)\oplus (E^*\otimes E)
\oplus (\wedge^2 E^*\otimes L).
\] 
The half-spin representations can be defined as vector bundles as:
\[
S_+(E,L):=\bigoplus_{k\geq 0}(\wedge^{2k}E\otimes L^{-k+e_+}), \quad
S_-(E,L):=\bigoplus_{k\geq 0}(\wedge^{2k+1}E\otimes L^{-k+e_-}),
\]
where $e_+=\lfloor\frac{e}{2}\rfloor$ and $e_-=\lfloor\frac{e-1}{2}\rfloor$. 
The Lie algebra bundle $\mathfrak{so}(V)$ acts naturally by exterior product (by elements of $E$) and contraction (by elements of $E^*$).

\subsection{$(E_6,\alpha_1)$}
\label{E6alpha1}
The action of $Spin_{10}$ on a half-spin representation $\Delta_+=V_{\omega_5}$ has only one proper orbit closure, 
the cone $Y$ over the spinor variety $S_{10}$. The ideal sheaf of this orbit closure has a beautiful minimal resolution \cite{RanestadSchreyer, KW12}:
\begin{multline}
	\label{resS10}
	A\leftarrow V_{\omega_1}\otimes A(-2) \leftarrow  V_{\omega_5}\otimes A(-3)\leftarrow \\ 
	\leftarrow V_{\omega_4}\otimes A(-5)\leftarrow V_{\omega_1}\otimes A(-6) \leftarrow A(-8) \leftarrow 0.
\end{multline}
In particular $Y$ is Gorenstein although its natural resolution (by the total space of $\cO(-1)$ 
over $S_{10}$) is not crepant.

In the relative setting, suppose we have a vector bundle $E$ of rank five on $X$ and a line bundle $L$. 
Then the rank ten vector bundle $V=E\oplus (E^*\otimes L)$ has a natural $L$-valued non-degenerate quadratic 
form, and the associated spin bundles are 
\[
S_+(E,L)=L^2\oplus (\wedge^2E\otimes L)\oplus \wedge^4E, \quad S_-(E,L)=(E\otimes L^2)\oplus (\wedge^3E\otimes L)\oplus \wedge^5E.
\]
As the maps in \eqref{resS10} are induced by the natural maps \eqref{refNaturalMaps}, the relative version of \eqref{resS10} yields a complex whose maps are induced by the natural maps
\[
V\otimes S_+(E,L)\longrightarrow S_-(E,L), \quad V\otimes S_-(E,L)\longrightarrow S_+(E,L)\otimes L
\]
and by the equalities $S_+(E,L)=S_-(E,L)^*\otimes L^2\otimes \det E$ and $V = V^* \otimes L$. For instance, the first map in \eqref{resS10} has relative version induced by
\begin{equation}
	\label{eqweightsSpin10}
	V \otimes S_+(E,L) \rightarrow S_-(E,L) = S_+(E,L)^*\otimes L^2\otimes \det E,
\end{equation}
which restricts to a map $Sym^2 S_+(E,L) \otimes V \otimes L^{-2} \otimes \det(E^*) \rightarrow \cO_X$. Twisting its dual by $ L^{-6}(\det E^*)^4$ and using the fact that $V=V^*\otimes L $, we get the last map $L^{-6} (\det E^*)^4 \rightarrow V\otimes L^{-5} (\det E^*)^3$ (see the complex \eqref{relativeVersionS10} below).

As a consequence, if $S_+(E,L)$ is globally generated and $s$ is a general section, then the ODL $D_Y(s)$ has a resolution given by
\begin{multline}
	\label{relativeVersionS10}
	\cO_{X} \leftarrow V^*\otimes L^{-1}(\det E^*) \leftarrow  S_+(E,L)\otimes L^{-3}(\det E^*)^2\leftarrow \\ 
	\leftarrow
	S_-(E,L)\otimes L^{-5}(\det E^*)^3\leftarrow V\otimes L^{-5} (\det E^*)^3  \leftarrow L^{-6} (\det E^*)^4 \leftarrow 0.
\end{multline}

We deduce: 

\begin{proposition}
	\label{S10ODL}
	Let $E, L$ be respectively a rank five vector bundle and a line bundle on a smooth variety $X$. Suppose that the associated spinor bundle 
	$S_+(E,L)$ is globally generated, and let $s$ be a general section. For $Y$ the cone over the spinor variety $S_{10}$, we have that $\dys$ has canonical rational singularities and
	\[
	\codim_X D_{Y}(s)=5, \qquad \codim_{D_{Y}(s)}\Sing (D_{Y}(s))=11, 
	\]
	\[
	K_{D_Y(s)}=(K_X\otimes L^{6}\otimes (\det E)^4)_{|D_Y(s)}.
	\]
\end{proposition}

The codimension of the singular locus of $\dys$ is equal to $11$ because $Y$ is the cone over the smooth ten-dimensional variety $S_{10}$. A Chern classes computation based on \eqref{relativeVersionS10} yields the following formula for the class of $\dys$ in terms of the Chern classes $e_i,l$ of $E$ and $L$:
\begin{multline*}
	[\dys]={e}_{1}^{3} {e}_{2}-{e}_{1}^{2} {e}_{3}+{e}_{1} {e}_{4}-{e}_{5}+2 {e}_{1}^{4} l+2
	{e}_{1}^{2} {e}_{2} l+2 {e}_{4} l+8 {e}_{1}^{3} l^{2}+\\{}+2 {e}_{1} {e}_{2} l^{2}+2 {e}_{3}
	l^{2}+12 {e}_{1}^{2} l^{3}+2 {e}_{2} l^{3}+8 {e}_{1} l^{4}+2 l^{5}.
\end{multline*}

Proposition \ref{S10ODL} above can be applied to construct new smooth fourfolds with trivial canonical bundle. Letting $L=\cO_X$, we need 
a nine-dimensional variety $X$ and a rank five vector bundle $E$, generated by global sections, such that 
$K_X^{-1}=(\det E)^4$. Examples of such pairs are 
\begin{equation*}
	\begin{split}
		X=\PP^3\times \PP^3\times \PP^3, & \qquad  E=p_1^* \cO(1)\oplus p_2^* \cO(1)\oplus p_3^* \cO(1)\oplus 2\cO_X, \\
		X=\PP^3\times \PP^3\times \PP^3, & \qquad  E=p_1^*\cQ_1\oplus p_2^* \cO(1)\oplus p_3^* \cO(1), \\
		X=\Gr(2,4)\times \IGr(2,5), & \qquad  E=p_1^*\cQ_1\oplus p_2^*\cQ_2, \\
		X=\mathbf{I}_2\Gr(3,8), & \qquad  E=\cQ.
	\end{split}
\end{equation*}
Here $\mathbf{I}_2\Gr(3,8)$ denotes a bisymplectic Grassmannian, namely a subvariety of $\Gr(3,8)$ parametrizing three-planes that
are isotropic with respect to a pair of general skew-symmetric forms. The four loci are of Calabi--Yau type.

\subsection{$(E_6,\alpha_2)$}
\label{E6alpha2}
The action of $GL_6$ on $\wedge^3 V_6$ has three proper orbit closures $Y_1, Y_5, Y_{10}$ of codimension respectively one (a degree 4 hypersurface), five, ten. The orbit closure $Y_5$ is the subvariety of partially decomposable tensors, already studied in \cite{BFMT} (see also Examples \ref{partDec} and \ref{exampleCrep}). The orbit closure $Y_{10}$ is the cone over $\Gr(3,V_6)$, whose study leads to a particular case of Proposition \ref{ordGrassODL}.

\subsection{$(E_7,\alpha_1)$}
\label{E7alpha1}
The action of $Spin_{12}$ on a half-spin representation $\Delta_+=V_{\omega_6}$ has only three proper orbit closures, 
the sixteen-dimensional cone $Y_{16}$ over the spinor variety $S_{12}$, an invariant quartic hypersurface $Y_1$, and the singular locus $Y_7$ of 
this hypersurface, which has codimension seven in $\Delta_+$. Remarkably, they are all Gorenstein \cite[\textsection 2]{KW13}.

\medskip
The minimal resolution of $Y_7$ is \cite{KW13}
\begin{multline*}
	A\leftarrow V_{\omega_6}\otimes A(-3) \leftarrow  V_{\omega_2}\otimes A(-4)\leftarrow V_{2\omega_1}\otimes A(-6)\leftarrow  \\ 
	\leftarrow V_{2\omega_1}\otimes A(-8) \leftarrow  V_{\omega_2}\otimes A(-10) 
	\leftarrow V_{\omega_6}\otimes A(-11) \leftarrow A(-14) \leftarrow 0.
\end{multline*}

\smallskip
In the relative setting, we need a vector bundle $E$ of rank six and a line bundle $L$ on a variety $X$. 
Then the rank twelve vector bundle $V=E\oplus (E^*\otimes L)$ has a natural $L$-valued non-degenerate quadratic 
form, and the associated spin bundles are 
\begin{align*}
	S_+(E,L) &=L^3\oplus (\wedge^2E\otimes L^2)\oplus (\wedge^4E\otimes L)\oplus \wedge^6E, 
	\\
	S_-(E,L)&=(E\otimes L^2)\oplus (\wedge^3E\otimes L)\oplus \wedge^5E.
\end{align*}

There are natural maps $V\otimes S_-(E,L)\rightarrow S_+(E,L)$ and $V\otimes S_+(E,L)\rightarrow S_-(E,L)\otimes L$. 
Moreover we have $S_+(E,L)=S_+(E,L)^*\otimes L^3\otimes (\det E)$ and $S_-(E,L)=S_-(E,L)^*\otimes L^2\otimes (\det E)$.
The Lie algebra action is $\fso(V)\otimes S_\pm(E,L)\rightarrow S_\pm(E,L)$, where $\fso(V)=\wedge^2V\otimes L^*$. 
This induces a map 
\[
Sym^2S_+(E,L)\rightarrow \fso(V) \otimes L^3\otimes (\det E), 
\] 
and using the Killing form 
on $\fso(V)$ we get an induced quartic map 
\[
Sym^4S_+(E,L)\rightarrow L^6\otimes (\det E)^2.
\]
When $X$ is a point,  
this is the quartic invariant that defines the hypersurface $Y_1$, whose singular locus is $Y_7$. 

From those maps we can derive the relative version of the complex above, as done in Section \ref{E6alpha1}. Suppose that $S_+(E,L)$ is globally generated
and $s$ is a general section. Then the ODL $D_{Y_7}(s)$ has a resolution by a complex 
\begin{multline*}
	\cO_{X} \leftarrow S_+(E,L)\otimes L^{-6} (\det E^*)^2 \leftarrow  \fso(V)\otimes L^{-6} (\det E^*)^2\leftarrow \\
	\leftarrow  Sym^2V\otimes L^{-10} (\det E^*)^3\leftarrow
	Sym^2V\otimes L^{-13} (\det E^*)^4\leftarrow \\
	\fso(V)\otimes L^{-15} (\det E^*)^5\leftarrow 
	S_+(E,L)\otimes L^{-18} (\det E^*)^6 \leftarrow L^{-21} (\det E^*)^7 \leftarrow 0.
\end{multline*}
We deduce: 

\begin{proposition}
	Let $E, L$ be respectively a rank six vector bundle and a line bundle on a smooth variety $X$. Suppose that the associated spinor bundle 
	$S_+(E,L)$ is globally generated, and let $s$ be a general section. Then $D_{Y_7}(s)$ has canonical rational singularities and
	\[
	\codim_X D_{Y_7}(s)=7, \qquad \codim_{D_{Y_7}(s)}\Sing (D_{Y_7}(s))=9, 
	\]
	\[
	K_{D_{Y_7}(s)}=(K_X\otimes L^{21}\otimes (\det E)^7)_{|D_{Y_7}(s)}.
	\]
\end{proposition}

This can be applied to construct new smooth fourfolds with trivial canonical bundle. 
Letting $L=\cO_X$, we need an eleven-dimensional variety $X$ and a rank six vector bundle $E$, 
generated by global sections, such that 
$K_X^{-1}=(\det E)^7$. An example of such a pair, which gives a locus of Calabi--Yau type, is $X=\IGr(2,8)$, $E=\cQ$.

\medskip
Let us now consider $Y_{16}$, which  is the cone over the spinor variety $S_{12}$. In order to
determine the last term $A(-N)$ of the minimal free resolution of $A/I_{Y_{16}}$, where   $A=Sym(\Delta_+^*)$, we observe that the index of $S_{12}$ is $10$, hence by \eqref{Niota} $N=22$. 

In the relative setting, the relative canonical bundle of the associated ODL must be of the 
form $L^{\alpha} (\det E)^{\beta}$ for some integers $\alpha$ and $\beta$ to be determined. 
Since $N=22$, this line bundle must be a factor of $S_+(E,L)^{22}$, which implies that
$\alpha+3\beta=66$. Moreover the determinant of $S_+(E,L)$ is $L^{48}(\det E)^{16}$, so 
$\alpha=3\beta$. Hence $\beta=11$ and $\alpha=33$. We get:

\begin{proposition}
	Let $E, L$ be respectively a rank six vector bundle and a line bundle on a smooth variety $X$. Suppose that the associated spinor bundle 
	$S_+(E,L)$ is globally generated, and let $s$ be a general section. Then $D_{Y_{16}}(s)$ has canonical rational singularities and
	\[
	\codim_X D_{Y_{16}}(s)=16, \qquad \codim_{D_{Y_{16}}(s)}\Sing (D_{Y_{16}}(s))=16, 
	\]
	\[
	K_{D_{Y_{16}}(s)}=(K_X\otimes L^{33}\otimes (\det E)^{11})_{|D_{Y_{16}}(s)}.\]
\end{proposition}

\subsection{$(E_7,\alpha_3)$}
\label{E7alpha3}
This case corresponds to the action of $GL(V_2)\times GL(V_6)$ on $V_2\otimes \wedge^2V_6$, where $V_2$ and $V_6$ are 
vector spaces of dimension two and six, respectively. Following \cite{KW13} there are two orbit closures that are Gorenstein, of codimension four and ten; we denote them by $Y_4$ and $Y_{10}$. It turns out that both admit crepant resolutions, given by 
the total spaces of the vector bundles $E_{4}=V_2\otimes (\cU\wedge V_6)$ on $\Gr(2,V_6)$, and $E_{10}=V_2\otimes \wedge^2\cU$ on $\Gr(4,V_6)$, whose respective ranks are $18$ and $12$. In fact those orbits are projectively dual to one another and the two vector bundles are
orthogonal (once we identify $V_6$ with its dual, see also Remark \ref{dualCrepant}). We deduce by \cite[Proposition 2.7]{BFMT} that the projectivizations of $Y_4$ and $Y_{10}$ have 
canonical bundles
\[
K_{\PP (Y_4)}=\cO_{\PP (Y_4)}(-18), \qquad K_{\PP (Y_{10})}=\cO_{\PP (Y_{10})}(-12).
\]

\begin{lemma}
	The singular loci of $Y_4$ and $Y_{10}$ have codimension three and five, respectively. 
	\begin{proof}
		Consider the resolution $\pi_4$ of $Y_4$ by the total space of $E_4$. If $\theta\in  V_2\otimes \wedge^2V_6$ has two 
		preimages $(U,\theta)$ and $(U',\theta)$, with $U\cap U'$ non-zero, then $\theta$ belongs to 
		\[
		V_2\otimes ((U\wedge V_6)\cap (U'\wedge V_6))=V_2\otimes (U_1\wedge V_6+\wedge^2U_3),
		\]
		where $U_1=U\cap U'$ and $U_3=U+U'$. We deduce that the locus over which $\pi_4$ is not an isomorphism is 
		the image $Y_7$ of the (birational) Kempf collapsing of the bundle $V_2\otimes (\cU_1\wedge V_6+\wedge^2\cU_3)$ 
		over $\Fl(1,3,V_6)$. The total space of this bundle has dimension $2\times 6+11=23$. Morever the fiber 
		of $\pi_4$ over a general point of $Y_7$ is a projective line $\PP(U_3/U_1)$, which implies that $Y_7$ is 
		the singular locus of $Y_4$. 
		
		The case of $Y_{10}$ is similar.
	\end{proof}
\end{lemma}

\begin{proposition}
	Let $E, F$ be two vector bundles of rank two and six respectively on a smooth variety $X$, such that
	$E\otimes \wedge^2F$ is generated by global sections. If $s$ is a general section, the ODL $D_{Y_4}(s)$ and $D_{Y_{10}}(s)$ have codimension four and ten respectively, they have canonical rational singularities of codimension three and five respectively, and
	\[
	K_{D_{Y_4}(s)}=(K_X\otimes (\det E)^6 \otimes (\det F)^4)_{|D_{Y_4}(s)},
	\]
	\[
	K_{D_{Y_{10}}(s)}=(K_X\otimes (\det E)^9 \otimes (\det F)^6)_{|D_{Y_{10}}(s)}.
	\]
	\begin{proof}
		The only things to show are the formulas for the canonical bundles. By Section \ref{ODLandRes}, $D_{Y_4}(s)$ admits a rational resolution of singularities $\theta':\zero(\tilde{s}) \rightarrow D_{Y_4}(s)$, $\tilde{s}$ being the section induced by $s$ of the vector bundle $(\pi^* E \otimes \pi^* \wedge^2 F)/(\pi^* E \otimes (\mathcal{U}_\mathbb{G}\wedge \pi^* F))$ on $\mathbb{G}$, the Grassmann bundle $\pi:\Gr(2,F)\rightarrow X$. Its determinant can be computed as
		\begin{align*}
			\MoveEqLeft \det(\pi^* E \otimes \pi^* \wedge^2 F) \otimes \det (\pi^* E \otimes (\mathcal{U}_\mathbb{G}\wedge \pi^* F))^* = \\
			& = \det(\pi^* E)^{15} \otimes \det(\pi^* \wedge^2 F)^2 \otimes \det(\pi^* E^*)^{9} \otimes \det((\mathcal{U}_\mathbb{G}\wedge \pi^* F)^*)^2\\
			& = \pi^* \det(E)^6 \otimes \pi^* \det(F)^{10} \otimes \det(\wedge ^2 \mathcal{U}_\mathbb{G}^*)^2 \otimes (\det(\mathcal{U}_\mathbb{G} \otimes \mathcal{Q}_\mathbb{G})^*)^2\\
			& = \pi^* \det(E)^6 \otimes \pi^* \det(F)^{10} \otimes \det(\mathcal{U}_\mathbb{G}^*)^{10} \otimes \det(\mathcal{Q}_\mathbb{G}^*) ^4\\
			& = \pi^* \det(E)^6 \otimes \pi^* \det(F)^{6} \otimes \det(\mathcal{U}_\mathbb{G}^*)^6.
		\end{align*}
		Since $K_\mathbb{G}=\det(\mathcal{U}_\mathbb{G})^6 \otimes \pi^* \det(F^*)^{2} \otimes \pi^* K_X$ and $\theta'_* K_{\zero(\tilde{s})}=K_{D_{Y_4}(s)}$, the conclusion follows from the adjunction formula. Similar computations imply the result for $D_{Y_{10}}(s)$.
	\end{proof}
\end{proposition}

In order to construct smooth fourfolds with trivial canonical bundle from $Y_{10}$, we would need a variety 
$X$ of dimension $14$. Then we could choose as $E$ a trivial vector bundle and ask for a rank six vector bundle $F$ such that 
$K_X^{-1}=(\det F)^6$. Examples of such pairs are 
\begin{equation*}
	\begin{split}
		X=\Gr(3,6)\times \PP^5, & \qquad  F=p_1^*\cQ\oplus p_2^*\cO(1)\oplus2\cO_X, \\
		X=\Gr(2,6)\times \QQ^6, & \qquad  F=p_1^*\cQ\oplus p_2^*\cO(1)\oplus\cO_X, \\
		X=\Gr(4,8)\cap H_1\cap H_2, & \qquad  F=\cQ|_X\oplus 2\cO_X, \\
		X=\IGr(4,9), & \qquad  F=\cQ\oplus\cO_X \; \mathrm{or} \; F=\cU^*\oplus2\cO_X;
	\end{split}
\end{equation*}
the associated loci are all of Calabi--Yau type. 

\begin{remark}
	\label{remIHS}
	From $Y_4$ we can construct a family of fourfolds $Z$ starting from $X=\Gr(2,6)$, $E=\cU^*$, and $F$ trivial. A direct computation shows that $D_{Y_7}(s)$ for a general section $s$ is empty, and $Z$ turns out to be an irreducible holomorphic symplectic variety (IHS). Similarly, from $Y_{10}$, with $X=\Gr(2,9)$, $E=\cU^*$, and $F$ trivial, we get a family of varieties $Z'$ which are IHS. In fact, 
	the two families consist in Hilbert schemes of length two subschemes of a general K3 surface of degree fourteen. 
	
	In order to understand why, notice that, for $Z$, the section $s$ defining the ODL belongs to $\Hom(\CC^6, \wedge^2 V_6)$; in general this 
	map is injective and its image defines a $\PP^5$ inside $\PP(\wedge^2 V_6)$. Let $C$ be the intersection of this $\PP^5$ with the 
	Pfaffian cubic in $\PP(\wedge^2 V_6)$. A point $U$ of $X=\Gr(2,V_6)$ belongs to $Z$
	if there is a two-dimensional subspace $V$ of $V_6$ such that the image of $U$ by $s$ is contained in $V\wedge V_6$. But this exactly 
	means that the Pfaffian cubic contains $s(U)$. From this observation one easily concludes that the ODL $Z$ coincides with 
	the Fano variety of lines in $C$. By \cite{BeauvilleDonagi}, this Fano variety is isomorphic to the Hilbert scheme of length
	two subschemes of the orthogonal K3 surface of degree fourteen. 
	
	The case of $Z'$ is similar, since the section $s'$ defines a $\PP^8$ inside $\PP(\wedge^2 {V_6})^*$, whose orthogonal is again a $\PP^5$ inside $\PP(\wedge^2 V_6)$.
\end{remark}

\subsection{$(E_7,\alpha_6)$}
\label{E7alpha6}
This case corresponds to the action of $GL(V_2)\times Spin_{10}$ on $V_2\otimes\Delta_+$, where $V_2$ is a two-dimensional
vector space and $\Delta_+$ is one of the half-spin representations of $Spin_{10}$. Following \cite{KW13}, the two 
orbit closures $Y_8$ and $Y_{13}$, of codimension eight and thirteen, are Gorenstein. Moreover the boundary of $Y_8$ is the codimension one orbit closure $Y_9$. Since $Y_8$ is normal, its singularities must be contained in the boundary of $Y_9$, which is $Y_{13}\cup Y_{15}$. 
In particular the  singular locus of $Y_8$ has  codimension at least five. 

The singularities of $Y_8$ and $Y_{13}$ are resolved, respectively, 
by the total space of the vector bundle $E_8=V_2\otimes S_+$ on $\QQ^8$, and by the total space of the vector bundle 
$E_{13}=V_2\otimes \cT_+$ on $\OGr(3,10)$, where
$S_+$ denotes the rank $8$ spinor bundle on $\QQ^8$, and $\cT_+$ the rank two half-spin bundle on $\OGr(3,10)$. The first one 
is crepant while the second is not. 

In the relative setting, suppose we are given vector bundles $E, F, L$ of respective ranks five, two and one, on a variety $X$. 
As in section \ref{aReminder}, we consider the quadratic bundle $V=E\oplus (E^* \otimes L)$ and the associated spinor bundles $S_\pm(E,L)$. 
Over the quadric bundle $\OGr(1,V)$, there exist spinor bundles $T_\pm$ of rank eight, which are subbundles
of the pullbacks $\pi^*S_\pm(E,L)$ via the projection $\pi : \OGr(1,V)\rightarrow X$. 

Suppose that $S_+(E,L)$ is globally generated, and that $s$ is a general section of $V_2\otimes S_+(E,L)$.  
According to \cite[Proposition 2.3]{BFMT}, 
the $Y_8$-degeneracy locus $D_{Y_8}(s)$ has a desingularization 
$\zero(\tilde{s})\subset \OGr(1,V)$ obtained as the zero locus of 
\[
\tilde{s}\in \HHH^0(\OGr(1,V), \pi^* F \otimes (\pi^*S_+(E,L)/T_+)).
\]
For $S_+(E,L)=L^2\oplus (\wedge^2E\otimes L)\oplus \wedge^4E$, we have
\begin{equation}
	\label{detS+}
	\det(S_+(E,L))=(\det E)^8\otimes L^{12}.
\end{equation}
In order to 
compute the determinant of $T_+$, we may restrict to the locus of lines $\ell$ contained in $E$. The image of the 
Clifford product by $\ell$ yields the subbundle with fiber
\[
T_{+, \ell} = (\ell\wedge E\otimes L)\oplus (\ell\wedge (\wedge^3E)) =  (\ell \otimes E/\ell\otimes L)\oplus (\ell\otimes \wedge^3(E/\ell)).
\]
This yields 
\begin{equation}
	\label{detT+}
	\det T_+=\pi^*((\det E)^4\otimes L^{4})\otimes \ell^4.
\end{equation}
Moreover, the relative tangent bundle is $T_{\OGr(1,V)/X}=\HOM(\ell, \ell^{\perp}/\ell).$
Since $\ell^{\perp}/\ell$ inherits the $L$-valued non-degenerate quadratic form,  its determinant is $\pi^*L^4$, and therefore
\begin{equation}
	\label{KOGR}
	K_{\OGr(1,V)/X}=\ell^{8}\otimes \pi^*L^{-4}.
\end{equation}
From the adjunction formula
\[
K_{\zero(\tilde{s})}=(K_{\OGr(1,V)/X}\otimes \pi^* K_X \otimes \det(\pi^*F\otimes (\pi^*S_+(E,L)/T_+)))_{|\zero(\tilde{s})}
\]
and \eqref{detS+}, \eqref{detT+}, \eqref{KOGR}, we get 

\begin{proposition}
	\label{propoY8}
	Let $E, F, L$ be three vector bundles of rank five, two, and one respectively on a smooth variety $X$, such that $F\otimes S_+(E,L)$ is generated by global sections. If $s$ is a general section, the ODL $D_{Y_8}(s)$ has codimension eight, it has canonical rational singularities of codimension at least five and
	\[
	K_{D_{Y_8}(s)}=(K_X \otimes \det(F)^{8}\otimes \det(E)^{8}\otimes L^{12})_{|D_{Y_8}(s)}.
	\]
\end{proposition}

\smallskip
As for $Y_{13}$, the twist $N$ of the last term $A(-N)$ of its free resolution can be computed as $N=c+d$, where $c$ is the codimension of the orbit closure, and $d$ is the degree of the numerator of the Hilbert series of the coordinate ring, see Section \ref{subsectConstructing}; we obtain that $N=20$.

In the relative setting, the relative canonical bundle $K_{D_{Y_{13}}(s)/X}$ must be of the 
form $\det(F)^{\alpha}\otimes \det(E)^{\beta}\otimes L^{\gamma}$. Being a (rational) power
of $\det(F\otimes S_+(E,L))$, it must be such that $\alpha=\beta=2\gamma/3$. Moreover it must
be a factor of the $20$-th tensor power of $F\otimes S_+(E,L)$, hence $\alpha=10$. We get:

\begin{proposition}
	Let $E, F, L$ be three vector bundles of rank five, two, and one respectively on a smooth variety $X$, such that $F\otimes S_+(E,L)$ is generated by global sections. If $s$ is a general section, the ODL $D_{Y_{13}}(s)$ has codimension thirteen, it has canonical rational singularities of codimension at least seven and
	\[
	K_{D_{Y_{13}}(s)}=(K_X \otimes \det(F)^{10}\otimes \det(E)^{10}\otimes L^{15})_{|D_{Y_{13}}(s)}.
	\]
\end{proposition}

\subsection{$(E_7,\alpha_7)$}
\label{E7alpha7}

This corresponds to the action of $E_6$ on $V_{\omega_1}$, the minimal representation of dimension $27$.
The non-trivial orbit closures are the cubic hypersurface $Y_1$, and the cone $Y_{10}$ over the Cayley plane $E_6/P_1$, 
which is also Gorenstein. Since the Cayley plane has index $12$, the last term in the minimal resolution of its ideal sheaf
is $\cO(-15)$.

The twenty-seven-dimensional representation of $E_6$ can be constructed from three vector spaces $V_1,V_2,V_3$ of dimension
three (see e.g.\ \cite{Mconf}). Suppose that generators of $\det V_1$, $\det V_2$, $\det V_3$ have been chosen.  
Then there is a natural $\mathbb{Z}_3$-graded Lie algebra structure on 
\[
\fe_6=\fsl(V_1)\times \fsl(V_2)\times \fsl(V_3)\oplus (V_1\otimes V_2\otimes V_3)\oplus (V_1^*\otimes V_2^*\otimes V_3^*)
\]
and a natural action of this Lie algebra on the graded module 
\[
V=\Hom(V_1,V_2)\oplus \Hom(V_2,V_3)\oplus \Hom(V_3,V_1). 
\]
The invariant cubic $I_3$ can then simply be expressed as $$I_3(x_1,x_2,x_3)=\det(x_1)+\det(x_2)+\det(x_3)-
c\,\mathrm{trace}(x_3\circ x_2\circ x_1),$$ 
for some constant $c$. 

In the relative setting, suppose given on a variety $X$ three vector bundles $E_1, E_2, E_3$ of rank three and line bundles $L, L_1, L_2, L_3$ such that 
\begin{equation}
	\label{condE6}
	\det E_1 = L \otimes L_2^{-1} \otimes L_3, \quad 
	\det E_2 = L_1  \otimes L \otimes  L_3^{-1}, \quad 
	\det E_3 = L_1^{-1} \otimes L_2 \otimes  L.
\end{equation}
In particular $(\det E_1 ) \otimes (\det E_2) \otimes (\det E_3)=L^3$. Then there is a natural Lie algebra 
structure on the vector bundle 
\[
\fe_6=\fsl(E_1)\times \fsl(E_2)\times \fsl(E_3)\oplus (E_1\otimes E_2\otimes E_3\otimes L^*)
\oplus (E_1^*\otimes E_2^*\otimes E_3^*\otimes L),
\]
and a natural structure of module over this bundle of Lie algebras on the vector bundle 
\begin{equation}
	\label{defOfV}
	V=\HOM(E_1,E_2\otimes L_3)\oplus \HOM(E_2,E_3\otimes L_1)\oplus \HOM(E_3,E_1\otimes L_2). 
\end{equation}
Note that the invariant cubic takes its values in $L_1L_2L_3$. 

Since $N=15$, $\dim V_{\omega_1}=27$ and $\det V=(L_1L_2L_3)^9$, we get the following

\begin{proposition}
	Let  $E_1, E_2, E_3$ and $L, L_1, L_2, L_3$ be vector bundles of ranks respectively three and one on a smooth variety $X$, 
	such that \eqref{condE6} holds. If $V$ defined in \eqref{defOfV} is globally generated and  $s$ is a general section, the ODL $D_{Y_{10}}(s)$ has codimension 
	ten, it has canonical rational singularities of codimension sixteen, and its canonical bundle is
	\[
	K_{D_{Y_{10}}(s)}=(K_X\otimes (L_1\otimes L_2\otimes L_3)^{5})_{|D_{Y_{10}}(s)}.
	\]
\end{proposition}

\subsection{$(E_8,\alpha_1)$}
\label{E8alpha1}
This case corresponds to the action of $Spin_{14}$ on the half-spin representation 
$\Delta_+=V_{\omega_7}$. This is the biggest among the parabolic spaces of exceptional type. According 
to \cite{KWE8}, six among the eight proper orbit closures are Gorenstein:
$Y_1$, $Y_5$, $Y_{10}$, $Y_{14}$, $Y_{29}$, and $Y_{42}$. The codimension of their singularities can be bounded from below by looking at the dimension of all the orbits.

The orbit closure $Y_5$ is the singular locus of the invariant octic hypersurface $Y_1$. It has a pure minimal resolution \cite{KWE8}
\begin{multline*}
	A\leftarrow V_{\omega_7}\otimes A(-7) \leftarrow  V_{\omega_2}\otimes A(-8)\leftarrow \\ 
	\leftarrow  V_{\omega_2}\otimes A(-12) \leftarrow V_{\omega_6}\otimes A(-13) \leftarrow A(-20) \leftarrow 0.
\end{multline*}
The two extreme maps in this complex are constructed from the invariant 
octic $Sym^8\Delta_+\rightarrow\CC$. Since $V_{\omega_2}$ is the adjoint representation, the next two maps are induced by the Lie algebra action $V_{\omega_2}\otimes \Delta_+\rightarrow \Delta_+$. Finally, the 
middle arrow can be defined from a map $Sym^4\Delta_+\rightarrow\fso_{14}$. 

In the relative setting, we need a rank seven vector bundle $E$ and a 
line bundle $L$ over a variety $X$. The quadratic vector bundle 
$V=E\oplus (E^*\otimes L)$ has two associated spin bundles 
\begin{align*}
	S_+(E,L)&=L^3\oplus (\wedge^2E\otimes L^2)\oplus (\wedge^4E\otimes L)\oplus \wedge^6E, 
	\\
	S_-(E,L)&=(E\otimes L^3)\oplus (\wedge^3E\otimes L^2)\oplus (\wedge^5E\otimes L)\oplus \wedge^7E.
\end{align*}
The invariant octic becomes a morphism
\[
Sym^8(S_+(E,L))\to L^{10}\otimes \det(E)^4,
\]
and the relative version of the previous complex can be computed as in Section \ref{E6alpha1}:
\begin{multline*}
	\cO_X\leftarrow S_+(E,L)\otimes L^{-10}(\det E^*)^4 \leftarrow \\
	\leftarrow \fso(V)\otimes L^{-10}(\det E^*)^4
	\leftarrow \fso(V)\otimes L^{-15}(\det E^*)^6\leftarrow \\
	\leftarrow
	S_+(E,L)^*\otimes L^{-15}(\det E^*)^6 \leftarrow L^{-25}(\det E^*)^{10}  \leftarrow 0.
\end{multline*}
\begin{proposition}
	Let $E, L$ be two vector bundles of rank seven and one respectively on a smooth variety $X$, such that $S_+(E,L)$ is generated by global sections. If $s$ is a general section, then the ODL $D_{Y_5}(s)$ has codimension five, it has canonical rational singularities of codimension five, and its canonical bundle is 
	\[
	K_{D_{Y_5}(s)}=(K_X\otimes L^{25}\otimes (\det E)^{10})_{|D_{Y_5}(s)}. \\ 
	\]
\end{proposition}

To obtain a smooth fourfold with trivial canonical bundle one can take $X=\PP^9$, $L=\cO_X$, $E=\cO(1)\oplus 6\cO_X$, which gives an ODL of Calabi--Yau type.

\medskip
Using the last term $A(-N)$ of the resolution of the remaining orbit closures, one can recover the canonical bundle of the respective ODL, as already done in Section \ref{E7alpha1}. The last term is $A(-28)$ for $Y_{10}$, $A(-32)$ for $Y_{14}$, $A(-44)$ for $Y_{29}$, $A(-52)$ for $Y_{42}$.

\begin{proposition}
	Let $L$ be a line bundle, and $E$ a vector bundle of rank seven on a smooth variety $X$. Suppose that $S_+(E,L)$ is generated by global sections
	and let $s$ be a general section. Then:
	\begin{itemize}[leftmargin=2.6ex]
		\item the ODL $D_{Y_{10}}(s)$ has codimension ten, it has canonical rational singularities of codimension at least four, and 
		its canonical bundle is 
		\[
		K_{D_{Y_{10}}(s)}=(K_X\otimes L^{35}\otimes (\det E)^{14})_{|D_{Y_{10}}(s)}; \\ 
		\]
		\item 
		the ODL $D_{Y_{14}}(s)$ has codimension fourteen, it has canonical rational singularities of codimension at least six, and its canonical bundle is 
		\[
		K_{D_{Y_{14}}(s)}=(K_X\otimes L^{40}\otimes (\det E)^{16})_{|D_{Y_{14}}(s)}; \\ 
		\]
		\item 
		the ODL $D_{Y_{29}}(s)$ has codimension twenty-nine, it has canonical rational singularities of codimension at least thirteen, and its canonical bundle is 
		\[
		K_{D_{Y_{29}}(s)}=(K_X\otimes L^{55}\otimes (\det E)^{22})_{|D_{Y_{29}}(s)}; \\ 
		\]
		\item 
		the ODL $D_{Y_{42}}(s)$ has codimension forty-two, it has canonical rational singularities of codimension twenty-two, and its canonical bundle is 
		\[
		K_{D_{Y_{42}}(s)}=(K_X\otimes L^{65}\otimes (\det E)^{26})_{|D_{Y_{42}}(s)}. \\ 
		\]
	\end{itemize}
\end{proposition}

\subsection{$(E_8,\alpha_2)$}
\label{E8alpha2}
This case corresponds to the action of $GL(V_8)$ on $\wedge^3V_8$, where $V_8$ is an eight-dimensional vector space. 
The four orbit closures $Y_1$, $Y_4$, $Y_{25}$, $Y_{40}$ are Gorenstein. The first one is the degree sixteen invariant hypersurface, while the last one is the cone over the Grassmannian $\Gr(3,8)$, whose study leads to a particular case of Proposition \ref{ordGrassODL}; moreover $Y_{25}$ is its tangent variety. As such, $Y_{25}$ is resolved by the total space of the rank sixteen vector 
bundle $E_{25}=\wedge^2\cU_3\wedge V_8$ over $\mathbb{G}=\Gr(3,8)$. A straightforward computation shows that $\det(E_{25})=\cO_\mathbb{G}(-8)=K_\mathbb{G}$, 
so that the resolution is crepant.

\begin{remark}
	\label{dualCrepant}
	Suppose that the total space of a homogeneous vector bundle $\cW$ over some $G/P$, which is a subbundle of the trivial bundle $\cV=G/P\times V$, defines a crepant 
	resolution of $Y\subset V$. This means that the projection to $V$ is birational
	and that $K_{G/P}=\det\cW$. Then the vector bundle $\cW^\perp=(\cV/\cW)^*$ is a subbundle 
	of $\cV^*=G/P\times V^*$, and verifies the crepancy condition $K_{G/P}=\det\cW^\perp$.
	Its image $Y^\perp\subset V^*$ therefore also admits a crepant resolution, in case 
	the projection from $\cW^\perp$ is again birational. 
	
	For example, applying this construction to $E_{25}$, and replacing $V_8$ by its 
	dual, we get the total space of the rank $40$ vector bundle $E_1=\wedge^2 \cU\wedge V_8$ over $\Gr(5,8)$, whose total space gives a crepant resolution of singularities 
	of the hypersurface $Y_1$. 
\end{remark}

\begin{lemma}
	The singular locus of $Y_{25}$ is $Y_{31}$.
	\begin{proof}
		We know from \cite{KWE8} that there are only three non-trivial orbits of codimension bigger than $25$:  $Y_{28}$,  $Y_{31}$, and  $Y_{40}$. Moreover $Y_{28}$ is not contained 
		in $Y_{25}$, because any point in $Y_{25}$ belongs to $\wedge^3V_6$ for some codimension two subspace $V_6\subset V_8$, while this is not the case for the general point of $Y_{28}$. Since $Y_{40}$ is 
		obviously contained in $Y_{31}$, we just need to prove that 
		$Y_{31}$ is contained in the singular locus of $Y_{25}$.

		Consider the crepant resolution $\pi_{25}$ of $Y_{25}$ by the total space of $E_{25}$. Let $\theta=e_{123}+e_{145}\in\wedge^3V_8$ be a representative of the open orbit in $Y_{31}$. Then $U_1(\theta)=\langle e_1\rangle$ and $U_5(\theta)=\langle e_1,e_2,e_3,e_4,e_5\rangle$
		are uniquely defined by $\theta$. Moreover $\theta$ belongs to $\wedge^2U_5(\theta)\wedge U_1(\theta) \simeq 
		\wedge^2(U_5(\theta)/U_1(\theta))\otimes U_1(\theta)$, so that $\theta$ defines, up to constant, a non-degenerate skew-symmetric two-form $\omega$ on 
		$U_5(\theta)/U_1(\theta)$.  It is then straightforward to check that $\theta$ belongs to $\wedge^2U_3\wedge V_8$ if and only
		if  $U_1(\theta)\subset U_3\subset U_5(\theta)$ and $U_3/U_1(\theta)$ is isotropic with respect to $\omega$. In particular, the fiber 
		of $\pi_{25}$ over a general point of $Y_{31}$ is a three-dimensional Grassmannian quadric. This 
		implies that $Y_{31}$ is contained in, hence equal to, the singular locus of $Y_{25}$. 
	\end{proof}
\end{lemma}

\begin{proposition}
	Let $E$ be a rank eight vector bundle on a smooth variety $X$, such that $\wedge^3E$ is generated by global sections. If $s$ is a general section, then the ODL $D_{Y_{25}}(s)$ has codimension twenty-five, it has canonical rational singularities of codimension six, and its canonical bundle is 
	\[
	K_{D_{Y_{25}}(s)}=(K_X\otimes (\det E)^{15})_{|D_{Y_{25}}(s)}.
	\]
\end{proposition}

The codimension four orbit closure $Y_4$ has a (non-crepant) 
desingularization given by the total space over the flag manifold
$\Fl(2,5,V_8)$ of the rank $31$ vector bundle $E_{4}=\wedge^3\cU_5+\cU_2\wedge \cU_5\wedge V_8$, where $\cU_2$ and $\cU_5$ are the 
tautological vector bundles. It has a remarkable minimal resolution \cite{KWE8}
\begin{multline*}
	A\leftarrow S_{(4^7,2)}V_8^*\otimes A(-10) \leftarrow  S_{(5^4,4^4)}V_8^*\otimes A(-12)\leftarrow \\
	\leftarrow S_{(7,5^7)}V_8^*\otimes A(-14)\leftarrow A(-24)\leftarrow 0,
\end{multline*}
where as usual $S_{(\lambda)}$ is the Schur functor corresponding to the Young diagram $\lambda$, e.g., $S_{(7,5^7)}V_8^*=(\wedge^8 V_8^*)^5\otimes \Sym^2 V_8^*$.

\begin{remark} Consider the rank $32$ vector bundle $E'_4= \wedge^2\cU_4\wedge V_8+\wedge^3 \cU_6$ 
	over the twenty-dimensional flag manifold $\Fl(4,6,V_8)$. A straightforward computation shows
	that $\det (E'_4)=K_{\Fl(4,6,V_8)}$. We claim that the total space of $E'_4$ maps surjectively 
	to $Y_4$. Indeed, according to \cite{KWE8} a representative of the open orbit in $Y_4$ 
	is 
	\[
	\theta= e_{157} +a e_{123}+b e_{124}+c e_{356}+d e_{456}+f e_{378}+g e_{478},
	\]
	for generic coefficients $a,b,c,d,f,g$. Then $\theta$ belongs to 
	$\wedge^2U_4\wedge V_8+\wedge^3 U_6$ for $U_4=\langle e_1,e_3,e_4,e_7\rangle$ and $U_6=\langle e_1,e_3,e_4,e_5,e_6,e_7\rangle$. But one could also choose 
	$U_4=\langle e_1,e_3,e_4,e_5\rangle$ and $U_6=\langle e_1,e_3,e_4,e_5,e_7,e_8\rangle$ or 
	$U_4=\langle e_3,e_4,e_5,e_7\rangle$ and $U_6=\langle e_1,e_2,e_3,e_4,e_5,e_7\rangle$. So the projection map from
	the total space of $E'_4$ to $Y_4$ is not birational (in which 
	case we would have obtained a crepant resolution of $Y_4$), 
	but only generically finite, probably of degree three. 
	The same phenomenon can be observed for the dual resolution 
	of $Y_4^{\perp}=Y_{12}$. 
\end{remark}

In the relative setting, the last term of the complex that resolves the structure sheaf of 
the ODL $D_{Y_{4}}(s)$ is a one-dimensional sub-$GL(E)$-module of $(\wedge^3E^*)^{24}$, so it must be $(\det E^*)^9$. Thus we have:

\begin{proposition}
	Let $E$ be a rank eight vector bundle on a smooth variety $X$, such that $\wedge^3E$ is generated by global sections. If $s$ is a general section, then the ODL $D_{Y_{4}}(s)$ has codimension four, it has canonical rational singularities of codimension at least two, and its canonical bundle is 
	\[
	K_{D_{Y_{4}}(s)}=(K_X\otimes (\det E)^{9})_{|D_{Y_{4}}(s)}.
	\]
\end{proposition}

Let $X=\PP^8$ and $E=\cO(1)\oplus 7\cO$; we obtain a (possibly singular) fourfold with trivial canonical bundle, which has Euler characteristic equal to two. 

\subsection{$(E_8,\alpha_6)$}
\label{E8alpha6}
This corresponds to the action of $GL(U_3)\times Spin_{10}$ on $U_3\otimes \Delta_+$, where $U_3$ is three-dimensional. 
Following \cite{KWE8}, the two orbit closures $Y_1, Y_9$ of codimension one (a degree twelve hypersurface) and nine are Gorenstein. The last term of the resolution of the ideal of $Y_{9}$ is $A(-24)$. Therefore, with the same arguments as in Section \ref{E7alpha6}, we obtain:

\begin{proposition}
	Let $E, F, L$ be three vector bundles of rank five, three, and one respectively on a smooth variety $X$, such that $F\otimes S_+(E,L)$ is generated by global sections. If $s$ is a general section, the ODL $D_{Y_9}(s)$ has codimension nine, it has canonical rational singularities of codimension at least two and
	\[
	K_{D_{Y_9}(s)}=(K_X \otimes \det(F)^{8}\otimes \det(E)^{12}\otimes L^{18})_{|D_{Y_9}(s)}.
	\]
\end{proposition}

\subsection{$(E_8,\alpha_7)$}
\label{e8alpha7}
This corresponds to the action of $GL(U_2)\times E_6$ on $U_2\otimes V_{\omega_1}$, where $U_2$ is two-dimensional and 
$V_{\omega_1}$ is the minimal representation of $E_6$, of dimension $27$ (see Section \ref{E7alpha7}).
Following \cite{KWE8}, the orbit closures $Y_1, Y_4, Y_7, Y_{25}$ of codimension one (a hypersurface of degree twelve), four, seven, and twenty-five are Gorenstein. The last three have singularities in codimension at least three, two, three respectively.

The minimal resolution of $Y_4$ is \cite{KWE8}
\begin{multline*}
	A\leftarrow S_{30}U_2^*\otimes A(-3) \leftarrow  (S_{33}U_2^*\oplus S_{51}U_2^*)\otimes A(-6)\leftarrow \\ \leftarrow S_{63}U_2^*\otimes A(-9) \leftarrow  S_{66}U_2^*\otimes A(-12) \leftarrow 0.
\end{multline*}
Since only the trivial representation of $E_6$ appears in this complex, it is entirely built from the four-dimensional 
degree three covariant  $S_{30}U_2\otimes I_3\subset Sym^3U_2\otimes Sym^3V_{\omega_1}\subset  Sym^3(U_2\otimes V_{\omega_1})$,
where $I_3\in Sym^3V_{\omega_1}$ is the invariant cubic. In fact it is 
defined by the Koszul complex that resolves the ideal of this covariant. Geometrically, the open orbit in $Y_4$ 
parametrizes injective maps from $U_2^*$ to $V_{\omega_1}$ whose image is contained in the cubic hypersurface. 
Moreover the singular locus of $Y_4$ has at least codimension three, since this is the minimal codimension of an 
orbit of smaller dimension. 

In the relative setting, we use the same model for $V_{\omega_1}$ as we did for  $(E_7,\alpha_7)$. We recall that it was
constructed, on the variety $X$, from three vector bundles $E_1, E_2, E_3$ of rank three, and line bundles $L, L_1, L_2, L_3$. 
We also need a vector bundle $F$ of rank two. 
The relative version of the cubic covariant $I_3$ takes its values in $L_1L_2L_3$, 
so that for $V$ defined as in \eqref{defOfV} and for a general section $s$ of $F\otimes V$, supposed as usual to be globally generated, 
the ODL $D_{Y_4}(s)$ is resolved by the complex 
\begin{multline*}
	\cO_X\leftarrow S_{30}F^*\otimes (L_1L_2L_3)^{-1} \leftarrow  (S_{33}F^*\oplus S_{51}F^*)
	\otimes (L_1L_2L_3)^{-2}\leftarrow \\
	\leftarrow S_{63}F^*\otimes (L_1L_2L_3)^{-3} 
	\leftarrow  S_{66}F^*\otimes (L_1L_2L_3)^{-4} \leftarrow 0.
\end{multline*}

\begin{proposition}
	Let $F$ be a rank two vector bundle on a smooth variety $X$ and assume that $E_1, E_2, E_3$ and $L, L_1, L_2, L_3$ are vector bundles of ranks respectively three, one on $X$ such that \eqref{condE6} holds. For $V$ defined as in \eqref{defOfV}, if $F \otimes V$ is globally generated and $s$ is a global section, the ODL $D_{Y_4}(s)$ has codimension four, it has canonical rational singularities of codimension at least three and its canonical bundle is
	\[
	K_{D_{Y_4}(s)}=(K_X\otimes (\det F)^6\otimes (L_1\otimes L_2\otimes L_3)^{4})_{|D_{Y_4}(s)}.
	\]
\end{proposition}

Since all representations of dimension one inside $V^{N}$ are of the form $\det(F)^{N/2} (L_1 L_2 L_3)^{\alpha}$ for a certain integer $\alpha$, and as $\det(V)=\det(F)^{27} \otimes (L_1L_2 L_3)^{18}$, we can compute the canonical bundle of the ODL in the relative setting.
Indeed, as the last term of the resolution is $A(-18)$ for $Y_7$ and $A(-36)$ for $Y_{25}$, we get:

\begin{proposition}
	Let $F$ be a rank two vector bundle on a smooth variety $X$ and assume that $E_1, E_2, E_3$ and $L, L_1, L_2, L_3$ are vector bundles of ranks respectively three, one on $X$ such that \eqref{condE6} holds. Let $V$ be as in \eqref{defOfV}. If $F \otimes V$ is globally generated and for a global section $s$, the ODL $D_{Y_7}(s)$ and $D_{Y_{25}}(s)$ have codimension seven, twenty-five respectively, they have canonical rational singularities of codimension at least two, three respectively and their canonical bundles are
	\[
	K_{D_{Y_7}(s)}=(K_X\otimes (\det F)^9\otimes (L_1\otimes L_2\otimes L_3)^{6})_{|D_{Y_7}(s)},
	\]
	\[
	K_{D_{Y_{25}}(s)}=(K_X\otimes (\det F)^{18}\otimes (L_1\otimes L_2\otimes L_3)^{12})_{|D_{Y_{25}}(s)}.
	\]
\end{proposition}

\subsection{$(E_8,\alpha_8)$}
\label{e8alpha8}

This corresponds to the action of the exceptional group $E_7$ on its minimal representation $V_{\omega_7}$ of 
dimension $56$. This exceptional representation can be conveniently described from $\fsl_8$ 
and its natural representation $V_8$  (see e.g. \cite{Mconf}):
\[
\fe_7=\fsl_8\oplus\wedge^4V_8, \qquad V_{\omega_7}=\wedge^2V_8\oplus\wedge^6V_8.
\]
There are three proper orbit closures in $V_{\omega_7}$:  the cone $Y_{28}$ over the 
Hermitian symmetric space $E_7/P_7$, a quartic hypersurface $Y_1$, and its singular locus $Y_{11}$. 
Remarkably, they are all Gorenstein \cite{KWE8}.

The intermediate orbit closure $Y_{11}$ has a resolution of singularities given by an irreducible rank twelve homogeneous vector bundle 
over the adjoint variety $E_7/P_1$, see \cite[Theorem 7.3]{LMmagic}. This resolution is not crepant, but by Theorem \ref{KempfInventiones} $Y_{11}$ is normal and has rational
singularities. In \cite{KWE8} the following minimal resolution
is conjectured:
\begin{multline*}
	A\leftarrow V_{\omega_7}\otimes A(-3) \leftarrow  V_{\omega_1}\otimes A(-4)\leftarrow V_{\omega_2}\otimes A(-7)\leftarrow  \\ 
	\leftarrow V_{\omega_6}\otimes A(-8) \leftarrow  V_{2\omega_7}\otimes A(-10) \leftarrow  V_{2\omega_7}\otimes A(-12) \leftarrow V_{\omega_6}\otimes A(-14)\leftarrow \\
	\leftarrow V_{\omega_2}\otimes A(-15)\leftarrow V_{\omega_1}\otimes A(-18)\leftarrow V_{\omega_7}\otimes A(-19)\leftarrow A(-22) \leftarrow 0.
\end{multline*}

By \cite{KWE8}, the numerator of the Hilbert series is monic. We deduce in particular that $Y_{11}$ is Gorenstein.

\smallskip
In the relative setting, we need a rank eight vector bundle $E$ on a variety $X$, and a line bundle $L$ such that
$L^2=\det E$. Then there is a natural Lie algebra structure of type $E_7$ on the bundle $\fsl(E)\oplus (\wedge^4E\otimes L^*)$, 
and a natural $E_7$-module structure on the rank $56$ vector bundle $V=\wedge^2E\oplus (\wedge^2E^*\otimes L)$.

The determinant of $V$ is equal to $L^{28}$ and its rank is $56$. Since the last term of the resolution of $Y_{28}$ (respectively $Y_{11}$) is $A(-38)$ (resp. $A(-22)$), in the relative situation the last term of the resolution of the ODL must be contained in $(V^*)^{38}$ (resp. $(V^*)^{22}$); as a consequence it is equal to $L^{-28\cdot 38/56}$ (resp. $L^{-28\cdot 22/56}$), and we obtain the following 

\begin{proposition}
	Let $E$ be a rank eight vector bundle on a smooth variety $X$, and let $L$ be a line bundle on $X$ such that
	$L^2=\det E$. If $V:=\wedge^2E\oplus (\wedge^2E^*\otimes L)$ is globally generated and for a global section $s$, the ODL $D_{Y_{11}}(s)$ and $D_{Y_{28}}(s)$ have codimension eleven, twenty-eight respectively, they have canonical rational singularities of codimension at least seventeen, twenty-eight respectively, and their canonical bundles are
	\[
	K_{D_{Y_{11}}(s)}=(K_X\otimes L^{11})_{|D_{Y_{11}}(s)},
	\]
	\[
	K_{D_{Y_{28}}(s)}=(K_X\otimes L^{19})_{|D_{Y_{28}}(s)}.
	\]
\end{proposition}

Note that the conditions that $\wedge^2E$ and $\wedge^2E^*\otimes L$ are both globally generated are, in a sense, opposed,
but not contradictory. For example, for a rank two vector bundle $F$ on $X$, let $E$ be the direct sum of four copies of $F$, and 
$L=(\det F)^2$. Then $\wedge^2E$ and $\wedge^2E^*\otimes L$ are both globally generated as soon as $F$ is. Of course
one could also twist $V$ by another line bundle $M$, at the price of an extra factor $M^{22}$ in $K_{D_{Y_{11}}(s)}$, see Section \ref{TwistedDeg}.

\subsection{$(F_4,\alpha_1)$}
\label{f4alpha1}
This case is very similar to the case of $(E_6,\alpha_2)$: it corresponds to the action of $Sp_6\times\CC^*$
on the fourteen dimensional representation $V_{\omega_3}$ of $Sp_6$, on which $\CC^*$ acts by homotheties. 
The non-trivial Gorenstein orbit closures are a quartic hypersurface $Y_1$ and the cone $Y_7$ over the Lagrangian 
Grassmannian $\IGr(3,6)$. Since the index of the latter is four, the last term in the minimal resolution of its ideal is $\cO(-10)$.
Indeed the minimal resolution computed in \cite{KW12} is
\begin{multline*}
	A\leftarrow V_{2\omega_1}\otimes A(-2) \leftarrow  V_{\omega_1+\omega_2}\otimes A(-3)\leftarrow \\
	\leftarrow V_{\omega_1+\omega_3}\otimes A(-4)
	\leftarrow V_{\omega_1+\omega_3}\otimes A(-6) \leftarrow  V_{\omega_1+\omega_2}\otimes A(-7) \leftarrow  \\ \leftarrow  V_{2\omega_1}\otimes A(-8) \leftarrow A(-10)\leftarrow 0.
\end{multline*}

In the relative setting, we need a rank six vector bundle $E$ with an everywhere non-degenerate skew-symmetric
form $\wedge^2E\ra L$. Note that this implies that $\det E = L^3$. Then the vector bundle $\wedge^{\langle 3\rangle}E:=\ker(\wedge^3E \rightarrow E\otimes L)$ has determinant $L^{21}$ and rank $14$; as in the relative situation the last term of the resolution of the ODL must be contained in $((\wedge^{\langle 3\rangle}E)^*)^{10}$, we deduce that it is equal to $L^{-21\cdot 10/14}$, and the following

\begin{proposition}
	Let $E$ be a rank six vector bundle on a smooth variety $X$, with an  everywhere non-degenerate skew-symmetric
	form $\wedge^2E\ra L$. Let $\wedge^{\langle 3\rangle}E$ denote the kernel of the induced map from 
	$\wedge^3E$ to $E\otimes L$. If $\wedge^{\langle 3\rangle}E$ is globally generated and $s$ is a general section, 
	the ODL $D_{Y_7}(s)$ has codimension seven, it has canonical rational singularities of codimension seven, and its canonical bundle is
	\[
	K_{D_{Y_7}(s)}=(K_X\otimes L^{15})_{|D_{Y_{7}}(s)}.
	\]
\end{proposition}


\makeatletter
\providecommand\@dotsep{5}
\makeatother
\listoftodos\relax

\end{document}